\documentclass[12pt,leqno]{amsart}
\usepackage{latexsym,amssymb,amsmath,amsfonts,enumerate}
\usepackage{graphicx}
\usepackage{ulem}
\usepackage{xcolor}

\numberwithin{equation}{section}

\newtheorem{theorem}{Theorem}[section]
\newtheorem{corollary}[theorem]{Corollary}
\newtheorem{lemma}[theorem]{Lemma}
\newtheorem{proposition}[theorem]{Proposition}

\theoremstyle{definition}

\numberwithin{equation}{section}

\newcommand{\ep}{\varepsilon}

\newcommand{\bi}{\begin{itemize}}
\newcommand{\ei}{\end{itemize}}
\newcommand{\ben}{\begin{enumerate}}
\newcommand{\een}{\end{enumerate}}

\newcommand\be{\begin{equation}}
\newcommand\ee{\end{equation}}
\newcommand\bea{\begin{eqnarray}}
\newcommand\eea{\end{eqnarray}}
\newcommand\beaa{\begin{eqnarray*}}
\newcommand\eeaa{\end{eqnarray*}}

\newcommand{\baa}{\begin{array}}
\newcommand{\eaa}{\end{array}}
\newcommand{\bss}{\begin{cases}}
\newcommand{\ess}{\end{cases}}

\def\bR{\mathbb{R}}

\def\epsilon{\varepsilon}

\begin{document}
\title[Spreading dynamics]{Spreading dynamics for diffusive competition systems\\ in shifting environments}

\author[T. Giletti]{Thomas Giletti}
\address{Laboratoire de Math\'{e}matiques Blaise Pascal, UMR 6620, University Clermont-Auvergne, 63178 Aubi\`{e}re, France}
\email{thomas.giletti@uca.fr}

\author[J.-S. Guo]{Jong-Shenq Guo}
\address{Department of Applied Mathematics and Data Science, Tamkang University, Tamsui, New Taipei City 251301, Taiwan}
\email{jsguo@mail.tku.edu.tw}

\medskip

\thanks{Date: \today. Corresponding author: J.-S. Guo.}

\thanks{The work of the second author (JSG) was supported in part by the National Science and Technology Council of Taiwan under the grant 113-2115-M-032-001.
This work is also supported in part by the CNRS-NCTS joint International Research Network ReaDiNet.}



\begin{abstract}
We study the spreading dynamics for two-species diffusive competition systems in shifting environments caused by climate changes. 
Our main goal is to derive conditions for extinction and persistence of each species in the case of a strong-weak competition. 
Depending on the pace of the climate change, but also on whether the strong competitor is faster or slower, we will uncover dramatically different outcomes in the asymptotic behavior of solutions. For instance, it may be that a weak and fast competitor survives while a strong and slow one does not. Furthermore, we find some parameter regime where the outcome depends not only on the climate change speed, but also on both species' resilience to it,
and even on the initial populations' distributions. Our results show how strikingly complex the effect of climate change may be on population dynamics as soon as several species are considered.
\end{abstract}

\maketitle


\section{Introduction}
\setcounter{equation}{0}

In this paper, we study the following diffusive Lotka-Volterra competition system in a shifting environment:
\be\label{P}
\bss
u_t=d_1u_{xx}+r_1u[\alpha_1(x-st)-u-a_2v ],\quad x\in\bR,\, t>0,\vspace{3pt}\\
v_t=d_2v_{xx}+r_2v[\alpha_2(x-st)-b_1u-v ],\quad x\in\bR,\, t>0,\\
\ess
\ee
where $u,v$ are two competing species and, for each $i=1,2$, $d_i$ stands for the diffusion coefficient and
$r_i\alpha_i$ is the intrinsic growth rate, which is spatially and temporally dependent.
Parameters $a_2$ and $b_1$ are inter-specific competition coefficients.
All constants $d_i,r_i,a_2,b_1$ in~\eqref{P} are assumed to be positive.

The functions $\alpha_1$ and $\alpha_2$ model the shifting of the habitat of each species, both moving with the same speed $s>0$. Throughout this paper, for each~$i$ we assume that~$\alpha_i$ is a continuously differentiable function defined in $\bR$ such that
\be\label{al}
-\infty<\alpha_i(-\infty)<0,\; \alpha_i(+\infty)=1,\; i=1,2.
\ee
In other words, the intrinsic growth rate of each species (that is without inter- and intra-species competition) is negative behind the moving frame with speed $s$ where the climate change occurs, and positive ahead. The right limits of the $\alpha_i$, which can be interpreted as the maximal capacities in the absence of the competitor and of climate change, are set to be~$1$ here. This is however without loss of generality up to some renormalization.

We refer the reader to the pioneering works~\cite{bdnz09,PotapovLewis} for the biological modeling viewpoint of such a system in the context of climate change. 
{Since then, the mathematical study of ecological models incorporating shifting environments has attracted a lot of attention. 
	This includes the existence and stability of forced waves (i.e. traveling waves with the same speed as the environmental shift), and the persistence vs extinction of each species.
	
	For the existence and stability of forced waves for a Fisher-KPP scalar equation, we refer the reader to \cite{br08,br09} in the case when the growth rate is negative outside of a spatially bounded domain. For situations where the growth rate remains positive either ahead or behind the shifting heterogeneity, we refer to~\cite{flw16} where it was motivated by an epidemiological model; see also~\cite{hz17}, and~\cite{bf18} which addressed the multiplicity and attractivity of forced waves in~$\bR$.
	For the studies on forced waves in reaction-diffusion systems, we only mention some works on competition systems.
	In particular, the existence, uniqueness, and stability of forced waves for a two-species competition-diffusion system in a shifting habitat was investigated in \cite{dll21}. 
	Then the existence of forced waves for two-species and three-species competition systems was studied in \cite{ggs23} including the existence of forced waves with critical speed. For the case of two competing species with nonlocal dispersal, we refer the reader to \cite{ww21} and the references cited therein.
	
	As for the work on persistence vs extinction, we refer to~\cite{bdnz09} and~\cite{lbsf14} for the Fisher-KPP equation in $\bR$ with respectively a spatially bounded or unbounded subdomain where the growth rate is positive. We also refer the reader to \cite{v15} for the scalar equation in a cylindrical or partially periodic domain, and to~\cite{coville} for the case of nonlocal dispersal. For more general nonlinearities (e.g. monostable or bistable), there seems to be scarcer results, but we may still refer to~\cite{bn15} and~\cite{bg19} where a sharp threshold with respect to initial data was established in some intermediate range of the climate change speed. Coming to the two-species competition systems, the gap formation between the species habitats was studied in \cite{bdd14}, and the persistence vs extinction of each species for the weak competition case (i.e. $a_2,b_1<1$) was studied in \cite{zwy17,ywz19,w22}. The case of strong competition (i.e. $a_2,b_1\ge 1$) was also addressed in \cite{ywz19}.
	The weak competition case with nonlocal dispersal was studied in the work \cite{wwz19}.
	For the study of spreading dynamics on predator-prey systems in shifting environments, we refer the reader to e.g.~\cite{cgg21,gsw23,acdg24}, and references therein.}

Throughout this work, the heterogeneous reaction-diffusion system~\eqref{P} will be supplemented with some initial data~$(u_0,v_0)$ satisfying at least the following assumption:
\be\label{i-bd}
\mbox{$(u_0,v_0)$ is continuous in $\bR$, and such that $0\le \not \equiv u_0,v_0\le 1$.}
\ee
In particular, by a comparison argument we immediately have that $0\le u,v \le 1$ for all~$t>0$. Furthermore, we will typically assume that $u_0$ and $v_0$ have a bounded support in the right direction, which is the most interesting case as far as spreading properties are concerned. Still, other situations may also be biologically relevant. As a matter of fact, the methods we develop here should allow one to handle the case when both or either species already populate the whole environment, e.g. the infimum limit of $u_0$ or $v_0$ as $x \to +\infty$ is positive. Another interesting question is the survival of both species when their habitats are also bounded (i.e. $\alpha_i (+\infty) < 0$), but this turns out to be significantly harder due to some of the nonlinear features of the Lotka-Volterra competition systems that our study will highlight. Coming back to the setting that we intend to study, that is with~\eqref{al} holding true and initial data having bounded support from above, roughly speaking this means that both species are in the process of invading their favorable environment (where $\alpha_i>0$) at the same time as climate change occurs.

In this setting, our goal throughout this paper will be to determine whether either species is able to outpace the climate change, and if so how fast ahead will it spread.\\

With this in mind, let us first discuss the homogeneous competition system, corresponding to~\eqref{P} ahead of the climate change, which is
\be\label{Ph}
\bss
u_t=d_1u_{xx}+r_1u(1-u-a_2v ),\quad x\in\bR,\, t>0,\vspace{3pt}\\
v_t=d_2v_{xx}+r_2v(1-b_1u-v),\quad x\in\bR,\, t>0 .\\
\ess
\ee
System~\eqref{Ph} always admits at least the three constant equilibria
\beaa
E_0:=(0,0),\;E_1:=(1,0),\; E_2:=(0,1).
\eeaa
In this paper, we will assume that
\begin{equation}\label{eq:hypothesis}
	a_2 < 1 < b_1 ,
\end{equation}
which is sometimes referred to as the strong-weak case, $u$ being the strong competitor and~$v$ the weak one. Indeed, as far as the underlying ODE system (i.e. homogeneous system \eqref{Ph} with no spatial component) is concerned, the steady state $E_1 = (1,0)$ is globally stable (in the subspace $\{u >0, v > 0\}$). This means that any initial population for~$u$, no matter how small, will eventually grow and drive $v$ to extinction. Up to inverting $u$ and $v$, assumption~\eqref{eq:hypothesis} of course covers the case when $a_2 > 1 > b_1$.


The spreading dynamics of the diffusive two-species competition system~\eqref{Ph} has been studied, in the strong-weak case~\eqref{eq:hypothesis}, in the remarkable work~\cite{gl19}. We also refer to~\cite{c18} for the spreading dynamics in the strong-strong case $a_2,b_1 >1$. Those works rely heavily on a comparison principle. As we will also build upon the construction of suitable upper-lower and lower-upper solutions, we recall this comparison principle in Section~\ref{sec:comparison}, and we describe in more detail the results of~\cite{gl19} in Section~\ref{sec:prelim_gl}. At this stage, let us provide some heuristics and introduce some relevant notation. 

As we will mostly deal with initial distributions $u_0$ and $v_0$ whose supports are bounded from above, the first relevant speeds are
\beaa
s_1^*:=2\sqrt{d_1r_1},\; s_2^*:=2\sqrt{d_2r_2}.
\eeaa
These are the minimal traveling wave speeds of the Fisher-KPP equations obtained by letting respectively $v \equiv 0$ in the $u$-equation and $u \equiv 0$ in the $v$-equation of \eqref{Ph}. In other words, these are the speeds of respectively $u$ in the absence of $v$, and~$v$ in the absence of $u$ (still omitting also the effect of climate change in both cases). In particular, they are also the fastest speeds each species can hope for.

This naturally suggests that the faster competitor, regardless of the strength of the competition, should persist in some fast enough moving frames. If the fast competitor is also the strong one, i.e. \eqref{eq:hypothesis} and $s_1^* > s_2^*$ simultaneously hold, then the $u$-component of the solution of \eqref{Ph} should spread with speed $s_1^*$, while the weaker $v$-component goes to extinction. On the other hand, the situation is more intricate when \eqref{eq:hypothesis} and $s_2^* > s_1^*$ hold. Indeed, one then expects the faster but weak competitor to spread first, and to eventually be replaced by its stronger counterpart. Therefore, two successive spreading fronts should emerge, connecting consecutively in time, $E_0  = (0,0)$ to $E_2 = (0,1)$, and $E_2 = (0,1)$ to $E_1 = (1,0)$. The main findings of Girardin and Lam~\cite{gl19} entail that, while the former front (i.e. the invading front for the fast species $v$) unsurprisingly moves with the asymptotic spreading speed $s_2^*$, the latter (i.e. the invading front for the slow species~$u$) behaves somewhat unexpectedly. They do establish that the latter front moves with an asymptotically constant spreading speed, which we will denote here by $s_{1,GL}^*$ (the subscript refering to the aforementioned authors), and whose formula we postpone to Section~\ref{sec:prelim_gl}. However, the value of~$s_{1,GL}^*$ may not be equal to~$2 \sqrt{d_1 r_1 (1-a_2)}$, which would have seemed like a natural candidate as the Fisher-KPP speed when plugging $v \equiv 1$ in the $u$-equation, or equivalently as the linear speed of invasion around the unstable steady state $E_2= (0,1)$.

For added clarity, we sum up the previous paragraph in the following theorem:
\begin{theorem}[\cite{gl19}, Theorem~1.1]\label{thm:gl}
	Assume that \eqref{eq:hypothesis} holds true, and let $(u,v)$ be a solution of the homogeneous Lotka-Volterra competition system~\eqref{Ph}, with compactly supported initial data $(u_0,v_0)$ satisfying~\eqref{i-bd}.
	\begin{enumerate}[$(i)$]
		\item If $s_1^* > s_2^*$, then
				$$\lim_{t \to +\infty} \sup_{x \in \mathbb{R}} v(x,t) =0,$$
				$$\lim_{t\to +\infty} \left\{ \sup_{|x| \geq (s_1^* + \ep) t} |u(x,t)| + \sup_{|x| \leq (s_1^* -\ep) t} |u(x,t)-1| \right\}= 0,$$
				for any $\ep \in (0, s_1^*)$.
			\item If $s_2^* > s_1^*$, then there exists $s_{1,GL}^* \in \big[ 2 \sqrt{d_1 r_1 (1-a_2)}, s_1^* \big]$ which does not depend on the initial data such that
		$$\lim_{t\to +\infty} \left\{ \sup_{|x| \leq (s_{1,GL}^* -\ep) t} |v(x,t)| +  \sup_{(s_{1,GL}^* + \ep) t \leq|x| \leq (s_2^* -\ep) t} |v(x,t)-1|
	+ \sup_{|x| \geq (s_{2}^* +\ep) t} |v(x,t) | \right\}= 0,$$
	$$\lim_{t\to +\infty} \left\{ \sup_{|x| \leq (s_{1,GL}^* -\ep) t} |u(x,t)-1| + \sup_{|x| \geq (s_{1,GL}^* +\ep) t} |u(x,t)| \right\}= 0,$$
	for any $\ep \in \left(0, \frac{ \min \{ s_{1,GL}^* , s_2^* - s_{1,GL}^* \}}{2} \right)$.
	\end{enumerate} 
\end{theorem}
Though the characterization of the spreading speed~$s_{1,GL}^*$ is also covered by Theorem~1.1 in~\cite{gl19}, 
here we postpone it to Theorem~\ref{thm:gl_add} below.\\

We can finally go back to the full climate change system~\eqref{P}. The question of large time spreading dynamics can be roughly broken down to the following: can each species keep pace with the climate change in order to survive? Surely, an intuitive answer at this stage should be to compare the climate change speed~$s$ with either $s_1^*$, $s_{1,GL}^*$ or $s_2^*$, depending on the parameter range and the investigated species. For instance, if $s_1^* < s_2^*$ and $s < s_{1,GL}^*$, one may expect both species to spread away from the climate change, and to recover similar dynamics as in Theorem~\ref{thm:gl}, with the obvious change that both species eventually die out behind the moving frame of the shifting heterogeneity. Our findings, which will be detailed below, are that, while we are able to confirm such heuristics in many cases, they may still fail in some parameter ranges due especially to nonlinear effects.

At last, the rest of this section will be dedicated to the statements of our main results. We will split those results into the two next subsections, dealing both with the strong-weak competition case where \eqref{eq:hypothesis} hold true, but respectively when $s_1^* > s_2^*$ and $s_2^* > s_1^*$. This disjunction of cases is motivated by the above discussion, which showed that they result into different spreading speeds even for solutions of the homogeneous system. 




\subsection{The case of a both stronger and faster competitor}\label{sec:strongfast}

Let us first assume that $s_1^* > s_2^*$, which turns out to be the simpler case. This means that $u$ is not only stronger, according to \eqref{eq:hypothesis}, but also faster. This suggests that $v$ will always lose to $u$, and the question mostly reduces to whether $u$ is able to outpace the climate change. This is answered in the following theorem, which is our first main result.

\begin{theorem}\label{thm:1_strongfast}
	Assume that $a_2 < 1 < b_1$, $s_1^* > s_2^*$, and let $(u,v)$ be the solution of \eqref{P}-\eqref{i-bd} such that the supports of $u_0$ and $b_0$ are both bounded from above.
	\begin{enumerate}[$(i)$]
		\item If $s \geq s_1^*$, then
					$$\lim_{ t \to +\infty } u(x,t) = \lim_{t \to \infty} v(x,t) = 0,$$
			uniformly with respect to $x \in \mathbb{R}$.
		\item If $s \in [s_2^*, s_1^*)$, then
				$$\lim_{t \to +\infty} v(x,t) = 0,$$
			uniformly with respect to $x \in \mathbb{R}$, and
				$$\lim_{t \to +\infty} \left\{\sup_{x\le (s - \ep)t} | u(x,t)| + \sup_{x\ge (s_1^*+\ep)t} |u(x,t)| \right\} = 0,$$ 
				$$\lim_{t \to +\infty} \left\{ \sup_{(s+ \ep) t \leq x \leq (s_1^* -\ep)t} | u(x,t)-1| \right\} = 0 ,$$
				for any $\ep  \in \left( 0, \frac{s_1^* - s }{2} \right)$.
		\item If $s \in [0,  s_2^* ) $, then
		$$\lim_{t \to +\infty} \left\{\sup_{x\le (s - \ep)t} |v(x,t)| + \sup_{x\ge (s +\ep)t} |v (x,t)| \right\} = 0,$$ 
		and
		$$\lim_{t \to +\infty} \left\{ \sup_{x\le (s -\ep)t} | u(x,t)| + \sup_{x\ge (s_1^*+\ep)t} | u(x,t) | \right\} = 0,$$ 
		$$\lim_{t \to +\infty} \left\{ \sup_{(s+ \ep) t \leq x \leq (s_1^* -\ep)t} | u(x,t)-1| \right\} = 0 ,$$
		for any $\ep  \in \left( 0, \frac{s_1^* - s }{2} \right)$.
	\end{enumerate}
\end{theorem}

Statement~$(i)$ simply states that, when the climate change is faster than both species, then both go to extinction in large time. The fact that extinction still occurs when $s = s_1^*$ comes from the logistic form of the growth rates, or in other words from the intra-species competition, which entails that (when the initial support is bounded from above) the solution slightly lags behind its asymptotic spreading speed even under the optimal conditions of no climate change and competitor.

According to statements~$(ii)$ and~$(iii)$, when $s < s_1^*$, then the species $u$ manages to spread ahead of the climate change, with its original speed $s_1^*$. On the other hand, $v$ is (mostly) driven to extinction. Notice that, while the cases $s< s_2^*$ and $s\in [s_2^*, s_1^*)$ may look similar, there is actually a small difference, which is that in the latter case, $v$ may still survive in the moving frame with speed $s$. This is because, under our assumptions, climate change does not have the same effect on the two species $u$ and $v$. Notice indeed that the functions~$\alpha_1$ and~$\alpha_2$ are a priori different. In particular, $v$ may be more resilient and may persist in a zone where~$u$ is already dying out due to climate change. Otherwise, $v$ may converge uniformly to~$0$ as $t \to \infty$ even if $s < s_2^*$. We will not explore further this dichotomy here, but we will uncover a similar phenomenon in the next subsection, where the influence of $\alpha_1,\alpha_2$ will be even more striking.


\subsection{The case of a stronger but slower competitor}\label{sec:strongslow}

In this subsection, we will assume that \eqref{eq:hypothesis} holds, that is $u$ is stronger at least without diffusion, yet $u$ has a slower spreading speed than~$v$, i.e.
$$s_1^* < s_2^* .$$
This situation will turn out to be much more complicated than when $u$ is faster. Before even going into our next main results, several important points need to be raised, and new notation needs to be introduced.\medskip

According to Theorem~\ref{thm:gl} and previous heuristics, the extinction or survival of the weak competitor~$v$ should be determined by whether $s > s_2^*$ or $s < s_2^*$. In the latter case, the stronger competitor~$u$, provided that it also outpaces the climate change, should eventually replace~$v$. Our knowledge of the homogeneous system~\eqref{Ph} also suggests that the invasion front of~$u$, connecting $E_1 = (1,0)$ and $E_2= (0,1)$, should move at the speed~$s_{1,GL}^*$, which we kept somewhat elusive so far. While we again refer to~\cite{gl19} for the full details and to Section~\ref{sec:prelim_gl} below for further comments, there are however two observations that need to be made here. 

First, under~\eqref{eq:hypothesis}, there indeed exist (a family of) traveling fronts of~\eqref{Ph} connecting the stable steady state $E_1 = (1,0)$ to the unstable steady state $E_2 = (0,1)$; see e.g.~\cite{k97}. Yet the minimal traveling front speed may not be linearly determined. By that, it is meant that the minimal traveling front speed may be strictly larger than $2 \sqrt{d_1 r_1 (1-a_2)}$, which is the speed associated with the linearized problem around the invaded state~$(0,1)$. In such a situation, and this will indeed be the case, we expect that also $s_{1,GL}^* > 2 \sqrt{ d_1 r_1 (1-a_2)}$. This is in contrast with the standard Fisher-KPP scalar equation, of which \eqref{Ph} may appear to be a generalization. This comes from the fact that, when putting system~\eqref{Ph} in a cooperative form by rewriting the system in terms of $(u,1-v)$, the sign of some nonlinear terms changes. 
{For more details on the issue of linear determinacy, we refer first to~\cite{h98,llw02,wll02} where it was introduced and studied in the context of systems. Since then, various sufficient or necessary conditions for linear determinacy have been found~\cite{h10,gl11}; see also~\cite{ao19} and the references therein for recent developments.
}

Second, a spreading front may be accelerated by the presence of a more favorable zone, even if this favorable zone moves further away. Indeed, in a monostable system, there are infinitely many traveling front speeds, among which the asymptotic spreading speed is selected depending on the initial distribution and its spatial decay. Usually, compactly supported initial data will naturally select the minimal speed. Yet, it may occur in the heterogeneous case that the spatial decay of the solution is modified by a far away favorable zone. This in turns modifies the asymptotic spreading speed and makes it strictly larger than the minimal traveling front speed. Due to the absence of the competitor~$v$ far away to the right, the slower species~$u$ may also go through this phenomenon, which led Girardin and Lam~\cite{gl19} to introduce the notion of ``nonlocal pulling''. We also refer to~\cite{hs14} for an earlier uncover, and to \cite{fgh22,GGM,ly22} for recent developments in the context of heterogeneous equations.\medskip

One may think that these difficulties could be swept under the rug here, using Theorem~\ref{thm:gl} and $s_{1,GL}^*$ as some sort of black box when dealing with the inclusion of climate change. This is actually not the case as these underlying features of~\eqref{Ph} also have an impact on the heterogeneous problem~\eqref{P}.

Indeed, the possibly nonlinear nature of the spreading speed $s_{1,GL}^*$ implies, very roughly, that a smaller population of~$u$ may not be as fast as a large one. In the homogeneous case, this is a rather meaningless observation as the spreading speed only appears in the large time behavior, where an initially small population will approach its maximal capacity anyway. However, here the climate change may prevent that, and in fact in the scalar monostable case~\cite{bg19} it has been shown that extinction and survival of a species may be dependent on the initial data. This is precisely due to the nonlinear determinacy of the minimal front speed. Naturally, the competition system will be prone to the same phenomenon.

This leads us to define yet another speed, roughly meant to represent the speed of a small population of individuals of species~$u$ in presence of~$v$. To do this, and keeping in mind the possibility of nonlocal pulling, we introduce the following heterogeneous scalar equation:
\begin{equation}\label{eq:shifting1}
	u_t = d_1 u_{xx} + r_1 u [ 1 - u - a_2 H (s_2^* t -x) ],
\end{equation}
where $H$ denotes the Heaviside step function, i.e. $H(\cdot)$ outputs $1$ on the positive half-line and $0$ on the negative half-line. Notice that we kept the intra-species competition term, making this equation nonlinear but of the Fisher-KPP type. This is merely for convenience since, though the linear equation is simpler, the Fisher-KPP type equation shares the same minimal speed. It is in this context that most results in the literature are already stated (including~\cite{ly22} whose results we will invoke below). In particular, it is known that the solutions of~\eqref{eq:shifting1} with initially compact support spread with some speed, which we denote by $s_{1,nlp}^*$ where $nlp$ stands for ``nonlocally pulled''. That is,
$$\limsup_{x \geq (s_{1,nlp}^* + \ep) t } \sup u(x,t)  = 0,$$
$$\limsup_{0 \leq x \leq (s_{1,nlp}^* - \ep) t } \sup |u(x,t) -1|= 0,$$
for any $\ep \in \left(0,s_{1,nlp}^* \right)$. More details, along with an explicit formula for $s_{1,nlp}^*$, will be given in Theorem~\ref{thm:nlp} below. At this stage, we merely point out that $$ 2 \sqrt{d_1 r_1 (1-a_2)} \leq s_{1,nlp}^* \leq s_{1,GL}^* ,$$
where both inequalities may or may not be equalities depending on parameters. \\

We are finally in a position to state our second main result.
\begin{theorem}\label{thm:2_strongslow}
Assume that $a_2 < 1 < b_1$, $s_2^* > s_1^*$, and let $(u,v)$ be the solution of \eqref{P}-\eqref{i-bd} such that the supports of $u_0$ and $v_0$ are both bounded from above.
\begin{enumerate}[$(i)$]
	\item If $s \geq s_2^*$, then
	$$\lim_{ t \to + \infty } u(x,t) = \lim_{t \to +\infty} v(x,t) = 0,$$
	uniformly with respect to $x \in \mathbb{R}$.
	\item If $s \in [s_1^*, s_2^*)$, then
	$$\lim_{t \to +\infty} u (x,t) = 0,$$
	uniformly with respect to $x \in \mathbb{R}$, and
	$$\lim_{t \to +\infty} \left\{\sup_{x\le (s - \ep)t} |v(x,t)| + \sup_{x\ge (s_2^*+\ep)t} | v (x,t)| \right\} = 0,$$ 
	$$\lim_{t \to +\infty} \left\{ \sup_{(s+ \ep) t \leq x \leq (s_2^* -\ep)t} | v (x,t)-1| \right\} = 0 ,$$
	for any $\ep  \in \left( 0, \frac{s_2^* - s }{2} \right)$.
	\item If $s \in (s_{1,GL}^*, s_1^*)$, then
	$$\lim_{t \to +\infty} \left\{ \sup_{x \leq (s -\ep) t} |u(x,t)| + \sup_{x \geq (s+ \ep)t } |u(x,t)| \right\} = 0,$$
	and
	$$\lim_{t \to +\infty} \left\{\sup_{x\le (s - \ep)t} |v(x,t)| + \sup_{x\ge (s_2^*+\ep)t} | v (x,t)| \right\} = 0,$$ 
	$$\lim_{t \to +\infty} \left\{ \sup_{(s+ \ep) t \leq x \leq (s_2^* -\ep)t} | v (x,t)-1| \right\} = 0 ,$$
	for any $\ep  \in \left( 0, \frac{s_2^* - s }{2} \right)$.
	\item If $s \in  [0,  s_{1,nlp}^* ) $, then
	$$\lim_{t \to +\infty} \left\{\sup_{x\le (s - \ep)t} | v(x,t)| + \sup_{(s + \ep)t \le x \le ( s_{1,GL}^* - \ep ) t} |v(x,t)| + \sup_{x\ge (s_2^* +\ep)t} |v (x,t)| \right\} = 0,$$ 
	$$\lim_{t \to +\infty} \left\{\sup_{ (s_{1,GL}^* + \ep ) t \le  x\le (s_2^* - \ep)t} | v (x,t) -1| \right\} = 0,$$ 
	for any $\ep \in \left(0 , \min \left\{\frac{s_{1,GL}^* - s}{2}, \frac{s_2^* - s_{1,GL}^* }{2} \right\} \right)$, and
	$$\lim_{t \to +\infty} \left\{ \sup_{x\le (s -\ep)t} | u(x,t)| + \sup_{x\ge (s_{1,GL}^*+\ep)t} | u(x,t)| \right\} = 0,$$ 
	$$\lim_{t \to +\infty} \left\{ \sup_{(s+ \ep) t \leq x \leq (s_{1,GL}^* -\ep)t} | u(x,t)-1| \right\} = 0 ,$$
	for any $\ep  \in \left( 0, \frac{s_{1,GL}^* - s }{2} \right)$.
\end{enumerate}
\end{theorem}
As in Theorem~\ref{thm:1_strongfast}, statement~$(i)$ merely states that, if the climate change is faster than both species, then they go to extinction as $t \to +\infty$. Next, statement~$(ii)$ confirms that, when the climate change is faster than the stronger but slower species~$u$, then it benefits to the weaker but faster species~$v$. Thus the latter manages to persist on its own. It even spreads with its natural speed $s_2^*$ corresponding to the situation when~$u$ is absent.

Statements~$(iii)$ and $(iv)$ are more intricate. In both cases, the climate change is slower than each species taken separately. In particular, the fast species~$v$ manages to outrace both the climate change and its competitor and persists in some intermediate moving frames. Then, in the case of statement~$(iii)$, the climate change is faster than the spreading speed of~$u$ in an environment populated by~$v$, and therefore it eventually drives~$u$ to extinction (except possibly in the moving frame with the exact speed~$s$). On the other hand, in the case of statement~$(iv)$, the stronger competitor~$u$ is always faster than the climate change. This is regardless of the initial data because, as explained above, $s_{1,nlp}^*$ is connected to the linearized problem and can be intuitively understood as the minimal speed of an arbitrarily small population of the species~$u$. Therefore, it manages to spread and eventually replaces the weaker competitor~$v$. The interface between the habitats of~$u$ and~$v$ moves at the same speed $s_{1,GL}^*$ as in the homogeneous competition system.

Theorem~\ref{thm:2_strongslow} leaves open the situation when $s \in ( s_{1,nlp}^*, s_{1,GL}^*)$, which we recall may or may not be empty depending on parameters. The reason is that, in this case, the climate change is too fast for a small population of $u$, but not necessarily for a large one, especially if on the other hand $v$ is small. Therefore, the outcome depends on the initial data in a non trivial way and there cannot be a general result as in all the previous cases. Moreover, the climate change functions also play a role.

The following result highlights this difficulty by showing that both spreading and extinction of $u$ may occur in this range.
\begin{proposition}\label{prop:diff_outcomes}
	Assume that $a_2 < 1 < b_1$ and $s_2^* > s_1^*$. Let $s \in (s_{1,nlp}^{*}, s_{1,GL}^*)$. Then:
	\begin{itemize}
		\item there exist $\alpha_1$, $\alpha_2$ and some compactly supported initial data $(u_0,v_0)$ such that
		$$\lim_{t \to \infty}\left\{\sup_{(s+\varepsilon) t \leq x \leq (s_2^* - \varepsilon) t} \Big[ | u(x,t)| + |v(x,t)-1| \Big] \right\} = 0,$$
		for any $\varepsilon \in \left(0, \frac{s_2^*-s}{2} \right)$;
		\item there exist $\alpha_1$, $\alpha_2$ and some compactly supported initial data $(u_0,v_0)$ such that $u$ spreads with speed $s_{1,GL}^*$, in the sense that
			$$\lim_{t \to \infty}\left\{\sup_{(s+\varepsilon) t \leq x \leq (s_{1,GL}^* - \varepsilon) t} \Big[ | u(x,t) - 1| + |v(x,t)| \Big] \right\} = 0,$$
			$$\lim_{t \to \infty}\left\{\sup_{(s_{1,GL}^* +\varepsilon) t \leq x \leq (s_{2}^* - \varepsilon) t} \Big[ | u(x,t) | + |v(x,t) -1| \Big] \right\} = 0,$$
	for any $\ep \in \left(0 , \min \left\{\frac{s_{1,GL}^* - s}{2}, \frac{s_2^* - s_{1,GL}^* }{2} \right\} \right)$.
	\end{itemize}
\end{proposition}

\subsection*{Plan of the paper} 
The rest of this paper is organized as follows. In \S~2, we shall provide some preliminaries on the spreading properties of solutions for scalar equations in shifting environments. From these known results for scalar equations, we will already be able to infer statements~$(i)$ and~$(ii)$ in both Theorems~\ref{thm:1_strongfast} and \ref{thm:2_strongslow},
some limits of $(iii)$ in Theorem~\ref{thm:1_strongfast} and some limits of $(iii)$, $(iv)$ in Theorem~\ref{thm:2_strongslow}.
Moreover, the classical nonlocal pulling phenomenon will be reviewed, and proved by a simple comparison argument different from that in~\cite{ly22}.

Then we provide in~\S~3 further preliminaries on the spreading dynamics for the homogeneous competition system~\eqref{Ph}, along with a comparison principle and a Liouville type result. We will go into more details in the results of~\cite{gl19}, the precise definition of the speed~$s_{1,GL}^*$, and some elements of their proofs. In particular we will produce some delicate upper-lower solutions, that will turn out to be useful to deal with the heterogeneous system~\eqref{P}. 

Finally, with all those tools in hand, the proofs of our main results shall be completed in~\S~4 and~\S~5. More precisely, \S~4 will deal with the case of a stronger and faster competitor, namely Theorem~\ref{thm:1_strongfast}. In~\S~5, we will turn to the more intricate case of a stronger but slower competitor, and prove both Theorem~\ref{thm:2_strongslow} and Proposition~\ref{prop:diff_outcomes}.



\section{Preliminaries : the scalar case}

Several preliminaries will be useful in our proofs, drawing from relevant references in the literature. In particular, we will present in more details the speed $s_{1,nlp}^*$, and later the speed~$s_{1,GL}^*$, that appeared in our main results.

In this section we will discuss known spreading properties of solutions of scalar equations with shifting heterogeneities. First, we will look at the effect of a climate change on a single species, drawing in particular from~\cite{lbsf14}. Putting this together with simple comparison arguments, this will allow us to already derive some estimates on the solutions of~\eqref{P}. Second, we will consider another type of moving heterogeneity, where the growth rate is everywhere positive, and explain the ``nonlocally pulled'' speed $s_{1,nlp}^*$ in the spirit of e.g.~\cite{hs14,gl19,fgh22,ly22}.

\subsection{The effect of climate change on a single species}\label{sec:prelim_single1}

In this subsection, we briefly discuss a first type of scalar equation with a shifting heterogeneity. This is simply the climate change problem with a single species, where the shifting heterogeneity changes sign with respect to the spatial variable. We will also discuss some consequences on our full system~\eqref{P}.

More precisely, let us denote by~$u$ the solution of
\be\label{U-eq}
\bss
u_t=d_1 u_{xx}+r_1 u[\alpha_1(x-st)- u],\;x\in\bR,\,t>0,\vspace{3pt} \\
u(x,0)=u_0(x),\; x\in\bR.
\ess
\ee
This kind of problem was studied in the context of climate change first in~\cite{bdnz09} when the favorable zone is bounded. Here the function $\alpha_1$ is positive on a right half-line and thus we instead invoke a result from~\cite{lbsf14}. We also recall that $s_1^* = 2 \sqrt{d_1 r_1}$, which is the spreading speed, or equivalently the minimal front speed, related to the homogeneous problem ahead of the climate change.
\begin{theorem}[\cite{lbsf14}, Theorems 2.1 and 2.2]\label{thm:lbsf}
	Let $u$ be the solution of the heterogeneous scalar equation~\eqref{U-eq} with continuous initial data $0 \leq \not \equiv u_0 \leq 1$. Then
	$$\lim_{t\to +\infty}\left\{\sup_{x\le (s-\ep)t} | u (x,t)| \right\}=0 ,$$
	for any~$\ep>0$, and the following two statements also hold true.
	\begin{enumerate}[$(i)$]
		\item If $s \geq s_1^*$ and the support of $u_0$ is bounded from above, then
		$$\lim_{t\to + \infty} u (x,t)= 0,$$
		uniformly with respect to~$x \in \mathbb{R}$.
		\item If $s < s_1^* $, then 
					$$\lim_{t\to +\infty}\left\{\sup_{(s+\ep) t \le x\le (s_1^*-\ep)t} | u (x,t) -1 |\right\}=0,$$
	for any $\ep \in \left(0, \frac{s_1^* - s}{2} \right)$, and if furthermore the support of $u_0$ is bounded from above, then
			$$\lim_{t\to +\infty}\left\{\sup_{x\ge (s_1^*+\ep)t} | u (x,t)| \right\}=0,$$
	for any $\ep >0$.
	\end{enumerate}
\end{theorem}
In particular, for compactly supported initial data, the population goes to extinction if $s \geq s_1^*$, and persists if $s < s_1^*$. That is, whatever the (compactly supported) initial distribution, $s_1^*$ is the critical climate change speed that the species can withstand. We insist on the fact that this is related to the Fisher-KPP type nonlinearity in~\eqref{U-eq}. For more general monostable nonlinearities, and more specifically when the minimal front speed is nonlinearly determined, the critical climate change speed depends on the initial data~\cite{bg19}.
\begin{proof}
	As a matter of fact, what we wrote here is a more general result than the one stated in~\cite{lbsf14}, where the case $s = s_1^*$ was not dealt with and moreover the climate change function~$\alpha$ was assumed to be nondecreasing. Yet this generalization does not raise any significant issue, thus we merely sketch a proof here.
	
	For the first limit in Theorem~\ref{thm:lbsf}, we introduce
	$$\overline{u} (x,t) = e^{-\delta t} + e^{\delta (x-st - K)} ,$$
	where $K \in \mathbb{R}$ is such that $\alpha_1 (\cdot) \leq \frac{\alpha_1 (-\infty)}{2} < 0$ on $(-\infty, K]$, and $\delta >0$ will be specified below. Then $\overline{u} \geq 1 \geq u$ on the parabolic boundary of the subdomain $\{ t \geq 0, \ x \leq st + K \}$. Moreover, in the same subdomain, we have that
\begin{eqnarray*}
	& & \partial_t \overline{u} - d_1  \overline{u}_{xx} - r_1 \overline{u} [ \alpha_1 - \overline{u}]\\
	& \geq & \partial_t \overline{u} - d_1 \overline{u}_{xx} -   \frac{r_1 \alpha_1 (-\infty)}{2} \overline{u}\\
	& \geq & e^{-\delta t} \left[ - \delta -  \frac{r_1 \alpha_1 (-\infty)}{2} \right] + e^{\delta (x-st- K)} \left[- s \delta - d_1 \delta^2 - \frac{r_1 \alpha_1 (-\infty)}{2} \right]\\
	& \geq & 0,
\end{eqnarray*}
where the last inequality holds provided that~$\delta$ is small enough. Thus, by the parabolic comparison principle, we infer that $u(x,t) \leq \overline{u} (x,t)$ for any $t \geq 0$ and $x \leq st +K$, and then that
\begin{equation}\label{eq:almost_Ueq}
	\forall \eta >0, \ \exists X >0, \quad \lim_{t \to +\infty} \left\{ \sup_{x \leq st - X } |u (x,t)| \right\} \leq \eta.
\end{equation}

Next, statement~$(i)$ (as well as the last limit in statement~$(ii)$) follows from the fact that
	$$u_t \leq d_1 u_{xx} +r_1 u,$$
	thus, by the comparison principle,
	$$u (x,t) \leq \frac{e^{r_1 t}}{\sqrt{4\pi d_1 t}} \int_{\mathbb{R}} e^{-\frac{(x-y)^2}{4d_1t}} u_0 (y) dy.$$
	Due to $u_0 \leq 1$, and if moreover the support of $u_0$ is included in the left half-line $(-\infty, X_0]$, we get that
	$$u (x,t) \leq \frac{e^{r_1 t}}{\sqrt{4\pi d_1 t}} \int_{-\infty}^{X_0} e^{-\frac{(x-y)^2}{4d_1t}} dy.$$
	Recalling that $s_1^* = 2 \sqrt{d_1 r_1} >0$, one may compute (using the standard Gaussian tail bound) that, for any $X >0$, there exist $C >0$ and $T>0$ large enough so that, for any $t \geq T$ and $x \geq s_1^* t  - X$, we have
	\begin{eqnarray*}
		u (x,t) 
		& \leq &  \frac{C e^{r_1 t}}{\sqrt{t}} e^{-\frac{(x-X_0)^2}{4d_1t}} \\
		& \leq &   \frac{C e^{r_1 t-\frac{(s_1^*t - X -X_0)^2}{4d_1t}}}{\sqrt{t}} \\
		& \leq &   \frac{C e^{ \sqrt{\frac{r_1}{d_1}} (X+X_0)}}{\sqrt{t}}, 
	\end{eqnarray*} 
	hence
	$$\lim_{t \to +\infty} \sup_{ x \geq s_1^* t - X} u (x,t)    = 0.$$
	When $s \geq s_1^*$, combining this with~\eqref{eq:almost_Ueq}, the uniform convergence to~$0$ follows.
	
	The other limit in statement~$(ii)$ is slightly more intricate. It relies on the fact that, for any $\delta \in (0,1)$, there exists $K>0$ large enough so that
	$$u_t \geq d_1 u_{xx} + r_1 u ( 1 - \delta - u),$$
	on the moving half-line $\{x \geq  st + K \}$. Moreover, this homogeneous Fisher-KPP equation admits compactly supported subsolutions with any speed (strictly) less than $s_1^*$; see~\cite{aw75}. Shifting such subsolution so that its support is included in $\{ x \geq st +K\}$, one may apply the comparison principle again on the same set to eventually conclude that, for any $s < s_1^*$ and $\varepsilon \in \left( 0, \frac{s_1^* - s}{2} \right)$, we have
	\begin{equation}\label{eq_last_detail}
		\liminf_{t\to +\infty}\left\{\inf_{(s+\ep) t \le x\le (s_1^*-\ep)t} u (x,t)  \right\} > 0 .
	\end{equation}
	Now, for any sequences $t_n \to +\infty$ and $(s+  \varepsilon) t_n \leq x_n \leq (s_1^*- \varepsilon)t_n$, then by standard parabolic estimates, up to extraction of a subsequence, the functions $u (x_n+ \cdot ,t_n+ \cdot )$ converge locally uniformly as $n \to +\infty$ to an entire in time solution~$u_\infty$ of the Fisher-KPP equation 
		$$\partial_t u_\infty = d_1 u_{\infty,xx} + r_1 u_\infty ( 1 - u_\infty).$$
	It follows from~\eqref{eq_last_detail} (where $\varepsilon$ can be made arbitrarily small) that $u_\infty \equiv 1$. Indeed, a Liouville type result (see, e.g., Proposition~1.14 in~\cite{bhn}) ensures that~$1$ is the only bounded entire in time solution with positive infimum. This completes the proof.
\end{proof}

Let us now go back to~\eqref{P}-\eqref{i-bd}. Due to $v \geq 0$, the $u$-component of the solution of~\eqref{P} is a subsolution to \eqref{U-eq} and we may apply the comparison principle. In particular, the extinction results on the solution of~\eqref{U-eq} immediately imply the extinction of~$u$, and a similar argument can be done on the $v$-component. This leads to the following corollary of Theorem~\ref{thm:lbsf}.
\begin{corollary}\label{thm:scalar_shifting1}
Let $(u,v)$ be the solution of \eqref{P} with initial data $(u_0,v_0)$ satisfying~\eqref{i-bd}. Then 
$$\lim_{t \to +\infty}\left\{ \sup_{x \leq (s-\varepsilon) t } |u(x,t)| + \sup_{x \leq (s-\ep )t } | v (x,t)|  \right\} = 0,$$
for any $\ep >0$, and also:
\begin{enumerate}[$(i)$]
	\item if the support of $u_0$, respectively $v_0$, is bounded from above, then
		$$\lim_{t\to +\infty}\left\{\sup_{x\ge (s_1^*+\ep)t} |u (x,t)|\right\}=0,$$
		respectively
	$$\lim_{t\to +\infty}\left\{\sup_{x\ge (s_2^*+\ep)t} |v (x,t)| \right\}=0,$$
	for any $\ep >0$;
	\item if $s \geq s_1^*$ and the support of $u_0$ is bounded from above, respectively $s \geq s_2^*$ and the support of $v_0$ is bounded from above, then 
	$$\lim_{t\to +\infty} \sup_{x \in \mathbb{R}} u (x,t) = 0,$$
	respectively 
	$$\lim_{t\to +\infty} \sup_{x \in \mathbb{R}} v (x,t) = 0.$$
\end{enumerate}
\end{corollary}
Notice that~\eqref{eq:hypothesis} is not necessary here. In fact, Corollary~\ref{thm:scalar_shifting1} holds for any $a_2,b_1>0$. Biologically, this means that no matter how strong the competition, in order to survive, each species must at the very least keep pace on its own with the shifting speed of the environment. 
This already proves some of the estimates of our main results, namely in Theorem~\ref{thm:1_strongfast}:~$(i)$, the first and second limits of~$(ii)$,
$$\lim_{t \to +\infty} \left\{\sup_{x\le (s - \ep)t} |v(x,t)| + \sup_{x\ge (s_2^* +\ep)t} |v (x,t)| \right\} = 0,$$ 
and the second limit of~$(iii)$; and in Theorem~\ref{thm:2_strongslow}:~$(i)$, the first and second limits of~$(ii)$, and some of the limits of~$(iii)$ and~$(iv)$.

The next result also follows from~\cite{lbsf14}, though it requires an additional step.
\begin{corollary}\label{th:general}
	Assume that $s_2^* \leq s<s_1^*$, and let $(u,v)$ be the solution of~\eqref{P} with initial data $(u_0,v_0)$ satisfying \eqref{i-bd} and such that the support of $v_0$ is bounded above. Then
\begin{equation*}
	\lim_{t\to +\infty}\left\{\sup_{(s+\ep)t\le x\le(s_1^*-\ep)t} |u(x,t)-1| \right\}=0,
\end{equation*}
	for all $\ep\in \left(0, \frac{s_1^*-s}{2} \right)$.
\end{corollary}
\begin{proof}
	Indeed, by~$(ii)$ of Corollary~\ref{thm:scalar_shifting1}, we know that in this case $v$ converges uniformly to~$0$. In particular, for any $\delta >0$, there exists $T>0$ such that
		$$u_t \geq d_1 u_{xx} + r_1 u [\alpha_1 (x-st) - \delta - u],$$
		for all $x \in \mathbb{R}$ and $t > T$. 
	
	If $\delta >0$ is small enough, then the shifting heterogeneity $\alpha_1 (\cdot ) -\delta$ still satisfies (up to some renormalization) assumption~\eqref{al}, so that Theorem~\ref{thm:lbsf} again applies. More specifically, by~$(ii)$ of~Theorem~\ref{thm:lbsf} and a comparison principle, we get that 
		$$\liminf_{t\to +\infty}\left\{\inf_{(s+\ep) t \le x\le ( 2\sqrt{d_1 r_1 (1-\delta)}-\ep)t}  u (x,t) - (1-\delta)  \right\} \geq 0,$$
		for any small $\ep >0$. Taking $\delta \to 0$, we reach the wanted conclusion.
\end{proof}
Putting this together with Corollary~\ref{thm:scalar_shifting1}, this completes the proof of~$(ii)$ in Theorem~\ref{thm:1_strongfast}. Moreover, notice that Corollary~\ref{th:general} again holds regardless of the type of the competition. In particular, one may invert the roles of $u$ and $v$ and get a similar result when $s_1^*\leq s < s_2^*$, which also ends the proof of~$(ii)$ in Theorem~\ref{thm:2_strongslow}. That is, when $a_2< 1 < b_1$ and $s_1^*\leq s<s_2^*$, the weak but faster competitor~$v$ benefits in some sense from the climate change, and spreads while the a priori stronger competitor does not.

\subsection{Shifting heterogeneities and nonlocal pulling}\label{sec:prelim_nlp}

In this subsection, we deal with another kind of shifting heterogeneity where basically the function~$\alpha$ is instead positive everywhere. As we explained in Section~\ref{sec:strongslow}, this may be relevant when~$u$ is slow and has to invade an environment which is partially populated by $v$. Due to $u$ being stronger, this means that $u$ spreads in an environment in which, putting aside the climate change, its growth rate is always positive but is higher far to the right where there is no competitor.

This led us to~\eqref{eq:shifting1}, that is
$$u_t = d_1 u_{xx} + r_1 u [1- u -a_2 H (s_2^* t -x)],$$
where $H$ is the Heaviside function. It arises when plugging $\alpha_1 =1$ everywhere, and $v = 1$ (respectively $v=0$) behind (respectively ahead of) the moving frame with speed $s_2^*$ in the $u$-equation of system~\eqref{P}. This is motivated by the eventuality that $v$ spreads with its optimal speed~$s_2^*$. In fact, since~$v$ may only behave in such a way asymptotically in time, and to deal with the climate change term, perturbations of~\eqref{eq:shifting1} will have to be considered, such as
\begin{equation*}\label{eq:shifting1delta}
	u_t = d_1 u_{xx} + r_1 u [ 1 - \delta - u - a_2 H ((s_2^* +\delta) t -x) -{a_2 \delta} H (x- (s_2^* +\delta) t) ],
\end{equation*}
for some small $\delta >0$. However, up to some renormalizations, both equations are equivalent so we will focus here on \eqref{eq:shifting1}.

Applying for instance the work~\cite{ly22}, but referring also to~\cite{fgh22} where a similar computation as below was done in the case when diffusion is heterogeneous instead, the following result is available.
\begin{theorem}[\cite{ly22}, Theorems~2 and~6]\label{thm:nlp}
	Let $u$ be the solution of the heterogeneous scalar equation~\eqref{eq:shifting1} with continuous initial data $0 \leq \not \equiv u_0 \leq 1$ whose support is bounded from above.
	
	Then there exists $s_{1,nlp}^*$ defined as follows:
	\begin{enumerate}[$(i)$]
        \item if $$s_2^* \leq s_1^*,$$
		then $$s_{1,nlp}^* = s_1^*;$$
		
		\item if $$ s_2^*  \in \left( s_1^* ,  2 \sqrt{d_1 r_1 (1- a_2)} + 2 \sqrt{d_1 r_1 a_2} \right) $$
		then 
$$s_{1,nlp}^{*}  = s_2^* - \frac{ \frac{(s_2^*)^2}{4} - d_1 r_1}{\frac{s_2^*}{2} -  \sqrt{d_1 r_1 a_2}} \in \left( 2 \sqrt{d_1 r_1 (1-a_2)}, s_1^* \right); $$
		\item if $$ s_2^* \geq 2 \sqrt{d_1 r_1 (1- a_2)} + 2 \sqrt{d_1 r_1 a_2}  $$
		then $$s_{1,nlp}^{*} = 2 \sqrt{d_1 r_1 (1-a_2)};$$
	\end{enumerate}
	such that $u$ spreads with speed $s_{1,nlp}^*$ in the following sense:
		$$\lim_{t \to +\infty} \sup_{x \geq (s_{1,nlp}^* + \ep ) t} u (x,t) = 0,$$
		$$\lim_{t \to +\infty} \sup_{ 0 \leq x \leq (s_{1,nlp}^* - \ep) t} |u (x,t) -1| = 0,$$
	for any $\ep \in \left(0, s_{1,nlp}^* \right)$.
\end{theorem}
Case~$(ii)$ is what is referred to as the nonlocal pulling phenomenon. Indeed, in that case the spreading speed is strictly less than $s_2^*$, i.e. the front remains in the zone where $H (s_2^* t -x) = 0$; yet the spreading speed is also strictly larger than $2 \sqrt{d_1 r_1 (1-a_2)}$. That is, the front moves faster than expected from its surrounding conditions. We point out that, for the sake of convenience, this spreading speed is denoted $s_{1,nlp}^*$ regardless of whether we are in the nonlocally pulled case.

The proof in~\cite{ly22} relies on a Hopf-Cole transform and a Hamilton-Jacobi method. Since it is more in line with our comparison arguments, let us give some key elements of an alternative approach using sub- and super-solutions in the spirit of~\cite{fgh22,GGM,gl19}. Those sub- and super-solutions will indeed play a role in the proofs of our main results. 

\begin{proof} For the sake of brevity, we omit the proof of case~$(i)$, where the solution spreads faster than the heterogeneity and basically the same sub- and supersolutions as in the homogeneous case (see e.g.~\cite{aw75}) can be used. Thus, we will assume below that
	$$s_2^* > s_1^* ,$$
which in any case is the only parameter range where the nonlocal pulling phenomenon occurs and that is relevant to our analysis.
	
	First of all, let us present the formal computation from which $s_{1,nlp}^*$ arises. As in the classical homogeneous Fisher-KPP case, it consists in looking for exponential ansatzes solving the linearized problem. Due to the heterogeneity, we will look for it in the form
	$$U_{ansatz} = e^{-\lambda(x-c t)},$$
	for $x \leq s_2^* t$, but
	$$U_{ansatz} = e^{-\lambda (s_2^* - c) t} \times e^{-\mu (x-s_2^* t)},$$
	for $x \geq s_2^* t$ (the first factor insures continuity as $x =s_2^* t$), where $c, \lambda,\mu > 0$. The minimal admissible $c$ will be the spreading speed $s^*_{1,nlp}$ we are looking for. Noticing that $H\geq 0$, a comparison principle directly entails that the spreading speed, if any, is necessarily less than~$s_1^*$. Keeping in mind that we only aim for a formal computation at this stage, we gain some time here and immediately accept that $c \leq 2 \sqrt{d_1 r_1} = s_1^*$ in the above ansatz.
	
	Plugging the ansatz into the linearization of~\eqref{eq:shifting1}, we get both
	$$d_1 \lambda^2 - c \lambda + r_1 (1-a_2)=0,$$
	$$d_1 \mu^2 - s_2^* \mu + r_1  + \lambda (s_2^* - c)= 0.$$
	For $\lambda$, we will pick the smallest positive root, that is
	$$\lambda= \frac{c - \sqrt{c^2 - 4 d_1 r_1 (1-a_2)}}{2d_1}.$$
	Its existence already entails that $c$ must be larger than $2 \sqrt{d_1 r_1 (1-a_2)}$.
	 
	Then for the existence of $\mu$, we find the condition
	$$ g(c):= (s_2^*)^2 - 4 d_1 \left( r_1 + \frac{c - \sqrt{c^2 - 4 d_1 r_1 (1-a_2)}}{2d_1} (s_2^* - c) \right) \geq 0 .$$
	It is straightforward to check that $g'(c) >0 $ at least in the relevant interval 
    $$c \in (2 \sqrt{d_1 r_1 (1-a_2)} , 2 \sqrt{d_1 r_1} ).$$ 
Thus, either 
			$$g(2 \sqrt{d_1 r_1 (1-a_2)}) \geq 0,$$
	in which case the minimal admissible speed for the ansatz is $s_{1,nlp}^{*} = 2 \sqrt{d_1 r_1 (1-a_2)}$, or
		$$ g(2 \sqrt{d_1 r_1 (1-a_2)}) < 0,$$
	in which case $s_{1,nlp}^{*}$ is the unique $c \in (2 \sqrt{d_1 r_1 (1-a_2)} , 2 \sqrt{d_1 r_1})$ such that
		$$ g (s_{1,nlp}^{*}) = 0.$$
	The condition $g(2 \sqrt{d_1 r_1 (1-a_2)}) = 0$ may be rewritten, but we omit the computation, as 
		$$s_2^* = 2 \sqrt{d_1 r_1 (1- a_2)} + 2 \sqrt{d_1 r_1 a_2} =:  \hat{s}.$$
Note that $g(2 \sqrt{d_1 r_1 (1-a_2)}) > 0$ if $s_2^* > \hat{s}$. Hence $s_{1,nlp}^{*}=2 \sqrt{d_1 r_1 (1-a_2)}$ if $s_2^*>\hat{s}$.
	
When $s_2^*<\hat{s}$, 
then $g(2 \sqrt{d_1 r_1 (1-a_2)}) < 0$ and another lengthy computation eventually leads one to $g (s_{1,nlp}^{*}) = 0$ with 
\be\label{gg1}
s_{1,nlp}^* = s_2^* - \frac{ \frac{(s_2^*)^2}{4} - d_1 r_1}{ \frac{s_2^*}{2} -  \sqrt{d_1 r_1 a_2}}.
\ee
Indeed, writing $c=s_2^*-\Gamma$, it follows from $g(c)=0$ that $\Gamma$ satisfies
\beaa
\left[\left(\frac{s_2^*}{2}\right)^2-d_1r_1a_2\right]\Gamma^2-s_2^*\left[\left(\frac{s_2^*}{2}\right)^2-d_1r_1\right]\Gamma
+\left[\left(\frac{s_2^*}{2}\right)^2-d_1r_1\right]^2=0.
\eeaa
Then the quadratic formula gives
\beaa
\Gamma_\pm=\frac{\left[\left(\frac{s_2^*}{2}\right)^2-d_1r_1\right]\left[\left(\frac{s_2^*}{2}\right)\pm \sqrt{d_1r_1a_2}\right]}{\left(\frac{s_2^*}{2}\right)^2-d_1r_1a_2}.
\eeaa
This implies that $s_{1,nlp}^*$ given in \eqref{gg1} is indeed the unique $c \in (2 \sqrt{d_1 r_1 (1-a_2)} , 2 \sqrt{d_1 r_1})$ such that
		$g (c) = 0$.

	In other words, $s_{1,nlp}^*$ as defined in Theorem~\ref{thm:nlp} is the minimal admissible speed for the ansatz~$U_{ansatz}$. Still, it remains to check that $s_{1,nlp}^*$ indeed predicts correctly the asymptotic spreading speed of the solution of~\eqref{eq:shifting1}.\medskip	
	
On the one hand, consider the case when $s_2^* \geq 2 \sqrt{d_1 r_1 (1-a_2)} + 2 \sqrt{d_1 r_1 a_2}$. Then one may take
			$$ c = 2 \sqrt{d_1 r_1 (1-a_2)}, \quad \lambda = \frac{c}{2d_1} , \quad  \mu = \frac{s_2^* + \sqrt{g(c)}}{2 d_1},$$
		and check that the resulting~$U_{ansatz}$ is a supersolution of~\eqref{eq:shifting1}. Indeed, the wanted differential inequality is satisfied both when $x < s_2^* t$ and $x > s_2^* t$. In order for $U_{ansatz}$ to be a generalized supersolution, it remains to verify that the left spatial derivative at $x =s_2^* t$ is larger than the right spatial derivative. This is equivalent to~$\mu \geq \lambda$, which is eventually confirmed after an elementary computation. 
		Up to a spatial shift, we may assume that $U_{ansatz}\geq u_0$, and infer by comparison that~$u$ spreads at most with speed~$2 \sqrt{d_1 r_1 (1-a_2)}$. 
		
Moreover, for any $c < 2 \sqrt{d_1 r_1 (1-a_2)}$, one may check that there exists $\omega, \eta_0>0$ such that, for any $ \eta \in (0,\eta_0)$,
			$$\underline{u} (x,t) = 
			\left\{
			\begin{array}{ll}
				\eta e^{-\frac{c}{2 d_1} (x-ct))} \sin ( \omega (x-ct)) & \mbox{ if } 0 < x -ct < \frac{\pi}{\omega}, \vspace{3pt}\\
				0 & \mbox{ otherwise},
			\end{array}
			\right.$$
	is a subsolution of \eqref{eq:shifting1}. As a matter of fact, this is simply a subsolution of a homogeneous Fisher-KPP equation, i.e.
	$$\underline{u}_t \leq d_1 \underline{u}_{xx} + r_1 \underline{u} (1-a_2 - \underline{u} ) .
	$$
	Next, by the strong maximum principle, $u (\cdot, t=1) >0$ and we can pick $\eta >0$ small enough so that $\underline{u} (\cdot, 1)\leq u (\cdot, 1)$. This proves that
				$$\forall c< 2 \sqrt{d_1 r_1 (1-a_2)} , \ \exists x \in \mathbb{R}, \quad \liminf_{t \to \infty} u (x+ct,t) > 0.$$
	One may further find that 
				$$\liminf_{t \to +\infty}  \inf_{ 0 \leq x \leq (2 \sqrt{d_1 r_1 (1-a_2)} - \varepsilon ) t} u (x,t)  >0 , $$
	and then that $u$ converges to~$1$ in the same sets as $t \to +\infty$, due to the fact that $1$ is the only entire in time solution of the Fisher-KPP equation which is bounded from both above and below by positive constants. We omit the details but refer to the second step of the proof of Proposition~\ref{prop:u_strongfast1} below for a similar argument.
	
On the other hand, when $s_1^* < s_2^*<2 \sqrt{d_1 r_1 (1-a_2)} + 2 \sqrt{d_1 r_1 a_2}$, one may take 
		$$ c= s_{1,nlp}^*, \quad \lambda= \frac{c - \sqrt{c^2 - 4 d_1 r_1 (1-a_2)}}{2d_1} \leq  \mu = \frac{s_2^*}{2d_1},$$
and again check that $U_{ansatz}$ is a supersolution.

	The subsolution is a little bit more intricate. We define, for $c \in \left( 2 \sqrt{d_1 r_1 (1-a_2)} , s_{1,nlp}^* \right)$,
	$$ \underline{u}_1 (x,t) :=  \max \{ 0, \eta_1 e^{-\lambda (x-ct) } -  \eta_2 e^{- (\lambda + \delta)  (x- ct)} \},$$
	where $\lambda$ is the smallest positive root of 
	$$d_1 \lambda^2 - c\lambda_1+ r_1 (1-a_2)=0,$$
	and $\delta >0$ is small so that $\lambda+ \delta$ lies below the other positive root. One may check that, with an appropriate choice of positive but arbitrarily small $\eta_1,\eta_2$, the resulting function satisfies
	$$\underline{u}_{1,t} \leq d_1 \underline{u}_{1,xx} + r_1 \underline{u}_1 (1 - a_2 - \underline{u}_1),$$
	thus it is a generalized subsolution of~\eqref{eq:shifting1}. Its exponential decay as $x \to +\infty$ prevents us, however, at this stage to compare it with the solution~$u$ of~\eqref{eq:shifting1}.
	
	This leads us to also define 
	$$\underline{u}_2 (x,t) :=  \left\{
			\begin{array}{ll}
		 \eta e^{- \lambda (s_2^* - c)t} \times e^{-\frac{s_2^*}{2d_1} (x-s_2^* t)} \sin \left(\omega (x-s_2^*t) \right) & \mbox{ if } 0 < x - s_2^* t < \frac{\pi}{\omega}, \vspace{3pt}\\
		0 & \mbox{ otherwise}.
	\end{array}
	\right.$$
	For a well-chosen $\omega$, and any $\eta >0$ small enough,
	this satisfies
	$$\underline{u}_{2,t} \leq d_1 \underline{u}_{2,xx} + r_1 \underline{u}_2 (1 - \underline{u}_2),$$
	and since its support is included in the subdomain~~$\{ x \geq s_2^* t\}$, it is also a subsolution of~\eqref{eq:shifting1}.
	
	Finally, an appropriate choice of $\eta$, $\eta_1$ and $\eta_2$ allows us to define a continuous function 
	$$\underline{u} (x,t) = \left\{ 
	\begin{array}{ll}
		\underline{u}_1 (x,t) & \mbox{ if } x - s_2^* t < 0 , \vspace{3pt}\\ 
		\max  \{ \underline{u}_1 (x,t), \underline{u}_2 (x,t) \} & \mbox{ if } 0 \leq x -s_2^* t \leq \frac{\pi}{2 \omega}, \vspace{3pt}\\
		\underline{u}_2 (x,t) & \mbox{ if } x-s_2^* t > \frac{\pi}{2\omega}.
\end{array}\right.$$	
 	By construction, it is a (generalized) subsolution of~\eqref{eq:shifting1}, whose spatial support at any given time is compactly supported. Thus, up to decreasing $\eta, \eta_1, \eta_2$, we may apply the comparison principle and conclude that the infimum limit of $u$ in the moving frame with speed $c$ is positive. As before, the convergence to~$1$ eventually follows. This ends this (sketch of a) proof of Theorem~\ref{thm:nlp}.
\end{proof}

\section{Preliminaries : the homogeneous competition system}

This section gathers further preliminaries, related to the study of the homogeneous two-species competition system. In particular we will go more in depth into the results of~\cite{gl19}, the spreading speed~$s_{1,GL}^*$ emerging in the competition system~\eqref{Ph}, and even some elements of their proof which we will make use of later. It turns out that system~\eqref{Ph} shares common features with the scalar equation with a shifting heterogeneity, and that $s_{1,GL}^*$ is defined in a similar way as~$s_{1,nlp}^*$ in the previous section. Finally, we will end this section with some useful Liouville type result on positive entire in time solutions of~\eqref{Ph}.

\subsection{A comparison principle for the competition system}\label{sec:comparison}

Let us first briefly recall that two-species competition systems satisfy a comparison principle. The results from the literature which we will state below make an extensive use of this principle, and so will we. For the sake of convenience, we state it here in the homogeneous case of~\eqref{Ph}, but it applies to the heterogeneous case of~\eqref{P}. 
\begin{proposition}
	Let $(\underline{u}, \overline{v})$ and $(\overline{u},\underline{v})$ be respectively a lower-upper and an upper-lower solutions of~\eqref{Ph}, that is,
$$
	\bss
	\underline{u}_t \leq d_1 \underline{u}_{xx}+r_1 \underline{u} [1 - \underline{u} -a_2 \overline{v} ],\quad x\in\bR,\, t>0,\vspace{3pt}\\
	\overline{v}_t \geq d_2 \overline{v}_{xx}+r_2 \overline{v} [1-b_1 \underline{u}- \overline{v} ],\quad x\in\bR,\, t>0,\\
	\ess
$$
and
$$
\bss
\overline{u}_t \leq d_1 \overline{u}_{xx}+r_1 \overline{u} [1- \overline{u} -a_2 \underline{v} ],\quad x\in\bR,\, t>0,\vspace{3pt}\\
\underline{v}_t \geq d_2 \underline{v}_{xx}+r_2 \underline{v} [1 -b_1 \overline{u}- \underline{v} ],\quad x\in\bR,\, t>0.\\
\ess
$$
If $ \underline{u} (t=0) \leq \overline{u} (t=0)$ and $\overline{v} (t=0) \geq \underline{v} (t=0)$, then the same inequalities hold true for any positive times.
\end{proposition}
This can be immediately derived from the comparison principle for cooperative systems by letting $w = 1 -v$ and rewriting the system in terms of~$u$ and~$w$. We omit the details.

\subsection{Spreading of solutions of the homogeneous competition system}\label{sec:prelim_gl}

Now we elaborate on the spreading Theorem~\ref{thm:gl}, not only by characterizing the speed $s_{1,GL}^*$ (see Theorem~\ref{thm:gl_add} below) but also by giving some key elements of the proof from \cite{gl19}. Indeed, because we cannot directly compare solutions of~\eqref{P} and~\eqref{Ph}, these elements will turn out to be necessary for the sake of our own arguments. \medskip 

First, as we mentioned in the introduction, when $s_1^* < s_2^*$ and both species have initially compact support, then the nonlocal pulling phenomenon may occur in the competition system. As a matter of fact, the heuristics is exactly the same as at the start of Section~\ref{sec:prelim_nlp}, where we merely replaced $v$ in the $u$-equation by a rough approximation of an invading front with speed $s_2^*$.
	
	In particular, when nonlocally pulled, the spreading speed $s_{1,GL}^*$ of the $u$-component of solutions of system~\eqref{Ph} turns out to follow the same formula as $s_{1,nlp}^*$. However, the nonlocally pulled case in~\eqref{Ph} occurs in a smaller parameter range. In short, the definition of $s_{1,nlp}^*$ in Theorem~\ref{thm:nlp} involved $2 \sqrt{d_1 r_1 (1-a_2)}$, which is the minimal speed of traveling fronts of the Fisher-KPP equation
		$$u_t = d_1 u_{xx} + r_1 u (1 -u - a_2),$$
	that is of the homogeneous equation behind the shifting heterogeneity in \eqref{eq:shifting1}. By analogy, to define $s_{1,GL}^*$ one should roughly replace $2 \sqrt{d_1 r_1 (1-a_2)}$ by the minimal speed of traveling front solutions of \eqref{Ph} connecting~$(1,0)$ and~$(0,1)$. The latter speed, whose existence was established by Kan-on in~\cite{k97}, will be denoted here by~$s_{1,mtf}^*$. It is sometimes equal to~$2 \sqrt{d_1 r_1 (1-a_2)}$, in which case $s_{1,mtf}^*$ is said to be linearly determined, and the corresponding traveling front is said to be pulled; in such a situation, then $s_{1,GL}^* = s_{1,nlp}^*$. However, it may be that $s_{1,mtf}^* > 2 \sqrt{d_1 r_1 (1- a_2)}$, in which case $s_{1,mtf}^*$ is said to be nonlinearly determined, and the corresponding traveling front is said to be pushed; in such a situation, then $s_{1,GL}^*$ may or may not differ from $s_{1,nlp}^*$, depending on whether nonlocal pulling also occurs (recall that $s_{1,nlp}^*$ denotes the spreading speed of~\eqref{eq:shifting1} regardless of whether nonlocal pulling actually happens). Let us highlight that one should distinguish, on the one hand, the minimal traveling front connecting $(1,0)$ and $(0,1)$, and on the other hand, the spreading front of~$u$ emerging from compactly supported initial data. For instance, it may be that the former is pushed while the latter is nonlocally pulled.

We refer again to~\cite{llw02} (and other references mentioned in the introduction section) for details on the linear determinacy or not of system \eqref{Ph}. Ultimately, one may rewrite the formulae in \cite{gl19} for the spreading speed $s_{1,GL}^*$ as follows.
	\begin{theorem}[\cite{gl19}, Theorem~1.1]\label{thm:gl_add}
		Assume that $s_2^* > s_1^*$ and let $s_{1,GL}^*$ be as in Theorem~\ref{thm:gl} above. Then
							$$s_{1,GL}^* = \max  \{ s_{1,nlp}^*, s_{1,mtf}^* \},$$
			where $s_{1,nlp}^*$ was defined in Theorem~\ref{thm:nlp} above, and $s_{1,mtf}^* \geq 2 \sqrt{d_1 r_1 (1-a_2)}$ is the minimal speed of traveling fronts of \eqref{Ph} connecting $(1,0)$ and $(0,1)$, i.e. the minimal $c$ such that there exists an entire in time solution $(u,v)(x,t)= (U,V) (x- ct)$ with
			$$U(-\infty) =1 > U (\cdot) > U (+\infty) = 0, \quad V(-\infty) = 0 < V (\cdot) < V (+\infty) = 1.$$
	\end{theorem}
	To summarize, the following cases may happen:
		\begin{itemize}
			\item $s_{1,GL}^* = s_{1,mtf}^* = s_{1,nlp}^* = 2 \sqrt{d_1 r_1 (1-a_2)}$, in which case the spreading front of $u$ may be said to be locally pulled;
			\item $s_{1,GL}^* = s_{1,mtf}^* > 2 \sqrt{d_1 r_1 (1-a_2)}$, in which case the spreading front of $u$ may be said to be pushed;
			\item $s_{1,GL}^* = s_{1,nlp}^* > s_{1,mtf}^*$, in which case the spreading front of $u$ may be said to be nonlocally pulled.
		\end{itemize}
	
	We hope that this theorem, together with Section~\ref{sec:prelim_nlp} and the above heuristics, already shed enough light on~$s_{1,GL}^*$ for the sake of understanding our results. However, as we mentioned above, our arguments will require us delving further into the proofs of Theorems~\ref{thm:gl} and~\ref{thm:gl_add}.
	
	Those proofs, while they share some similarities with that of Theorem~\ref{thm:nlp}, are a lot more intricate due to dealing with the two components of the system, and the possibility of nonlinear speeds, on top of the nonlocal pulling phenomenon. Thus we will only focus on some specific steps of relevance to us. Simplifying it to the extreme, the general strategy still involves the construction of sub- and supersolutions by gluing two ansatzes: on the one hand, a traveling wave solution of system~\eqref{P} with speed $s \geq s_{1,mtf}^*$; on the other hand, an exponential solution of the linearized problem around $(0,0)$ in the moving frame with speed~$s_2^*$. The characterization of $s_{1,GL}^*$ arises from a continuity matching condition.
	
	First, the fact that the $u$-component of solutions of~\eqref{Ph} spreads with speed at least~$s_{1,nlp}^*$ can be proved in the same way as Theorem~\ref{thm:nlp}, or even as its corollary. Indeed, by a relatively straightforward comparison argument, one may notice first that
		$$v (\cdot, \cdot) \leq 1, \qquad \limsup_{t \to +\infty} \sup_{|x| \geq (s_2^* + \ep) t} |v(x,t)| = 0,$$
	and plugging this information in the $u$-equation, that 
		$$u_t \geq d_1 u_{xx} + r_1 u [ 1-\delta - u - a_2 H ((s_2^* + \delta)t-x) - a_2 \delta H (x-( s_2^* + \delta) t)] ,$$
	for some arbitrarily small~$\delta >0$ and up to some shift. This equation is of the same type as~\eqref{eq:shifting1}, and applying Theorem~\ref{thm:nlp}, one gets a lower bound for the spreading speed of~$u$, which letting $\delta \to 0$ approaches $s_{1,nlp}^*$. 
	
	Next, when $s_{1,mtf}^* > s_{1,nlp}^*$, the fact that the $u$-component spreads at least with speed~$s_{1,mtf}^*$ actually traces back to~\cite{llw02,wll02}. Their proof relies on an abstract dynamical system approach, which as far as we know does not fit our need to adapt it to the spatially heterogeneous framework of~\eqref{P}. Thus, we suggest a different approach by a lower-upper solution, in the spirit of~\cite{rothe} for the scalar monostable case.
\begin{proposition}\label{prop:gl_subsol}
		{Assume that $s_{1,mtf}^* >s_{1,nlp}^*$.} 
Let $(U,V) (x- s_{1,mtf}^*t)$ be a traveling front with minimal speed~$s_{1,mtf}^*$ of \eqref{Ph}, connecting $(1,0)$ and $(0,1)$. Then:
		\begin{enumerate}[$(i)$]
			\item the functions $U,V$ are respectively decreasing and increasing in space;
			\item up to some shift, the function $U$ satisfies
                     $$U (z) \sim_{z \to + \infty} e^{-\lambda^* z},$$
                  	where $\lambda^*$ is the larger positive root of
						$$ d_1 (\lambda^*)^2 - s_{1,mtf}^* \lambda^* + r_1 (1-a_2)=0;$$
			\item define $\chi$ a smooth non-increasing function so that
					$$\forall z \leq 0, \ \chi (z) =1, \qquad \forall z \geq 1 , \ \chi (z) = 0,$$
					and 
					$$\underline{U} (z) = U (z ) - \delta_1 \chi ( z+M+1) - \delta_1 \chi ( M+1 -z ) e^{-\lambda z} ,$$
					$$\overline{V} (z) = V (z ) + \delta_1 \chi ( z+M+1) + \delta_2 \chi ( M+1 - z),$$
				then for well-chosen $M,\delta_1, \delta_2 >0$, $\lambda \in(0, \lambda^*)$, and $\ep >0$ arbitrarily small, the pair
					$$ \left(\underline{U} (x- (s_{1,mtf}^* - \ep) t), \overline{V} (x - (s_{1,mtf}^*-\ep)t ) \right)$$
				is a lower-upper solution of \eqref{Ph}.
		\end{enumerate}
	\end{proposition}
	Notice that in particular, on the one hand $\underline{U} \leq 1 -\delta_1$ and $V \geq \delta_1$ on a left half-line, and on the other hand $\underline{U}\leq 0$ and $\overline{V} \geq 1$ on a right half-line. Therefore, provided that a solution~$(u,v)$ tends to $(1,0)$, say locally uniformly, one may then apply a comparison principle after some potentially large time and on a right half-line, and find that the solution spreads with speed at least $s_{1,mtf}^* - \ep$, hence~$s_{1,mtf}^*$ by letting $\ep \to 0$. While we do not detail that part of the proof of Theorems~\ref{thm:gl} and~\ref{thm:gl_add} here, we will in fact proceed in this exact way for the spreading of solutions of the heterogeneous competition system~\eqref{P}.
	\begin{proof} The first two items are proved respectively in~\cite[Proposition~A.4]{gl19} and~\cite[Theorem~2.16]{XZ}. The fact that the function~$U$ selects the faster exponential decay comes from the assumption that $s_{1,mtf}^* > s_{1,nlp}^*$, which together with $s_{1,nlp}^* \geq 2 \sqrt{d_1 r_1 (1-a_2)}$ ensures we are here in the nonlinearly determined/pushed case.
		
	Thus we only prove the third item. For conciseness, in this proof we simply denote $c = s_{1,mtf}^*$. Recalling again that $s_{1,mtf}^* > s_{1,nlp}^* \geq 2 \sqrt{d_1 r_1 (1-a_2)}$, we may fix $\ep >0$ so that $$c- \ep  > 2 \sqrt{d_1 r_1 (1- a_2)}.$$ 
	Then we may fix $\lambda< \lambda^*$ such that
		$$d_1 \lambda^2 - (c- \ep ) \lambda + r_1  (1-a_2)  < 0 .$$
	Due to $a_2 < 1 < b_1$, we may pick $\delta_0\in(0,1)$ such that
			$$ \frac{1 +\delta_0 }{2 - a_2}, \frac{b_1  + 1 + \delta_0 b_1}{2 b_1} < 1, \quad \frac{1}{2} - \frac{b_1 \delta_0}{2 b_1 - 2}  >0.$$
	Then we pick $ M >1$ such that 
\be\label{M-}
U (z) \geq  \max \left\{ \frac{1+ \delta_0}{2- a_2} , \frac{b_1  + 1 + \delta_0 b_1 }{2 b_1}   \right\},\;
			V (z) \leq \frac{1}{2} - \frac{b_1 \delta_0}{2 b_1 - 2},\quad\mbox{ for $z \leq -M$,}
\ee
and
\be\label{M+}
U (z) \leq e^{-\lambda z},\;
			V (z) \geq  \max \left\{ \frac{3}{4},  1  + \frac{ d_1 \lambda^2 - (c - \ep) \lambda + r_1 (1-a_2) }{2 r_1 a_2} \right\},\quad
\mbox{ for $z \geq M$.}
\ee
Furthermore, up to increasing $M$, we will assume that
				$$ e^{-\lambda M}\leq \frac{- (d_1 \lambda^2 - (c - \ep ) \lambda + r_1 (1-a_2))}{8 r_1 a_2 b_1} .$$
	By a strong maximum principle, one may check that $U' < 0$ and $V'>0$, hence we may also let
			$$\eta := \inf \left\{ \inf_{|z| \leq M+1} V'(z) , \inf_{|z| \leq M+1} [-U'(z)] \right\} >0 .$$
	Finally, we let $\delta_1, \delta_2 \in (0,\delta_0)$ be such that
\be\label{delta12}
\bss
4 b_1 e^{-\lambda M} < \frac{\delta_2}{\delta_1} < \frac{ - (d_1 \lambda^2 - (c-\ep)\lambda + r_1 (1-a_2))}{2 r_1 a_2},\\
\delta_1,\delta_2 \leq \min\left\{
\frac{\eta \ep}{ d_1 \| \chi ''\|_\infty + 2 \lambda d_1 \| \chi '\|_\infty + d_1 \lambda^2  + 2r_1},
     \frac{\eta \ep}{ d_2 \| \chi ''\|_\infty  +  r_2 (2b_1+1)}\right\}.
\ess
\ee
	
We now check the two differential inequalities
			$$d_1 \underline{U} '' + (c- \ep) \underline{U} ' + r_1 \underline{U} (1 -\underline{U} - a_2 \overline{V} ) \geq 0,$$
			$$d_2 \overline{V} '' + (c- \ep) \overline{V} ' + r_2 \overline{V} (1 -\overline{V} - b_1 \underline{U} ) \leq 0.$$
	First consider $z \leq - M -1$. Then $\chi ( z + M + 1) =1$, $\chi(M+1-z)=0$ and
			\begin{eqnarray*}
				&& d_1 \underline{U} '' + (c-\ep) \underline{U} ' + r_1 \underline{U} (1 -\underline{U} - a_2 \overline{V} )\\
			& = & d_1 U'' + c U' + r_1 U (1- U - a_2 V) - \ep U' \\
			&& + r_1 \underline{U} (1 -\underline{U} - a_2 \overline{V} ) - r_1 U (1-U - a_2 V)\\
				& \geq &  r_1 \underline{U} (1 -\underline{U} - a_2 \overline{V} ) - r_1 U (1-U - a_2 V),
			\end{eqnarray*} 
			where we used the fact that $U' < 0$. Going further,
				\begin{eqnarray*}
				&& d_1 \underline{U} '' + (c-\ep) \underline{U} ' + r_1 \underline{U} (1 -\underline{U} - a_2 \overline{V} )\\
				& \geq &  r_1 ( U - \delta_1 )  (1 - U + \delta_1 - a_2 V - a_2 \delta_1)- r_1 U (1-U - a_2 V) \\
				& = & r_1 U (\delta_1 - a_2 \delta_1) - r_1 \delta_1 (1-U + \delta_1 - a_2 V - a_2 \delta_1) \\	
				& \geq & r_1 \delta_1\left[ (2 - a_2) U  - (1 + \delta_1) \right] \\
				& \geq & 0,
			\end{eqnarray*} 
			where the last inequality comes from our choice of $0 < \delta_1 \leq \delta_0 $ and~\eqref{M-}. 
Similarly, we have that, for $z \leq -M -1$, 
				\begin{eqnarray*}
				&& d_2 \overline{V} '' + (c-\ep) \overline{V} ' + r_2 \overline{V} (1 -\overline{V} - b_1 \underline{U} )\\
				& = & d_2 V'' + c V' + r_2 V (1- V - b_1 U) - \ep V' \\
				&& + r_2 \overline{V} (1 -\overline{V} - b_1 \underline{U} ) - r_2 V (1-V - b_1 U)\\
				& \leq &  r_2 ( V + \delta_1 )  (1 - V - \delta_1 - b_1 U + b_1 \delta_1)- r_2 V (1-V - b_1 U) \\
				& = & r_2 V (-\delta_1 + b_1 \delta_1)  + r_2 \delta_1  (1-V -\delta_1 - b_1 U + b_1 \delta_1) \\
				& \leq & r_2 \delta_1 \left[ (1 - b_1 + b_1 \delta_1 ) + b_1 (1-U)    + V (b_1 - 1) \right] \\
				& \leq & 0,
			\end{eqnarray*}
using \eqref{M-}, $0<\delta_1\le\delta_0$, $V\ge 0$ and $b_1>1$.

			Next, we consider $z \geq M+1$. Then $\chi (M+1 - z) = 1$, $\chi(z+M+1)=0$ and
			\begin{eqnarray*}
				&& d_1 \underline{U} '' + (c-\ep) \underline{U} ' + r_1 \underline{U} (1 -\underline{U} - a_2 \overline{V} )\\
				& = & d_1 U'' + c U' + r_1 U (1- U - a_2 V) - \ep U' \\
				&& - \delta_1 e^{-\lambda z} (d_1 \lambda^2 - (c-\ep) \lambda)\\
				&& + r_1 \underline{U} (1 - \underline{U}- a_2 \overline{V}) - r_1 U (1- U - a_2 V)\\
				& \geq  & - \delta_1 e^{-\lambda z} (d_1 \lambda^2 - (c-\ep) \lambda + r_1 (1- a_2)  )\\
				&& - r_1 \delta_1 e^{-\lambda z} ( a_2  - a_2 \overline{V} - \underline{U} ) +  r_1 U (  \delta_1 e^{-\lambda z} - a_2 \delta_2 )\\
				& \geq & - \delta_1 e^{-\lambda z} (d_1 \lambda^2 - (c-\ep) \lambda + r_1 (1- a_2)  )\\
				&& - r_1  a_2 \delta_1 e^{-\lambda z} (1 - V)  -  r_1 a_2 \delta_2 e^{-\lambda z} \\
				& \geq  & 0 ,
			\end{eqnarray*}
where we used~\eqref{M+} and~\eqref{delta12}. 
            As far as the other component is concerned, we compute for $z \geq M+1$ that
				\begin{eqnarray*}
				&& d_2 \overline{V} '' + (c-\ep) \overline{V} ' + r_2 \overline{V} (1 -\overline{V} - b_1 \underline{U} )\\
				& = & d_2 V'' + c V' + r_2 V (1- V - b_1 U) - \ep V' \\
				&& + r_2 \overline{V} (1 -\overline{V} - b_1 \underline{U} ) - r_2 V (1-V - b_1 U)\\
				& \leq &  r_2 ( V + \delta_2 )  (1 - V - \delta_2 - b_1 U + b_1 \delta_1 e^{-\lambda z})- r_2 V (1-V -b_1 U) \\
				& = & r_2 V (-\delta_2 + b_1 \delta_1 e^{-\lambda z})  + r_2 \delta_2 (1-V -\delta_2 - b_1 U + b_1 \delta_1 e^{-\lambda z}) \\
				& \leq & r_2 \delta_2 ( 1 - 2 V   - \delta_2  )  + r_2 V b_1 \delta_1 e^{-\lambda z}  + r_2 b_1 \delta_1 \delta_2 e^{-\lambda z}\\
				& \leq & r_2 \delta_2 ( 1 - 2 V    )  + 2 r_2 b_1 \delta_1 e^{-\lambda M}\\
				& \leq & - \frac{r_2 \delta_2}{2} + 2 r_2 b_1 \delta_1 e^{-\lambda M} \\
				& \leq & 0,
			\end{eqnarray*}
due to \eqref{M+} and the first inequality in \eqref{delta12}.
			
Finally, since $\underline{U}=U$ and $\overline{V}=V$ in $[-M,M]$, it remains to consider the interval 
$$(-M-1,-M)\cup(M,M+1).$$
For $-M-1< z < -M$, we have 
\beaa
\underline{U}(z)=U(z)-\delta_1\chi(M+1+z),\; \overline{V}(z)=V(z)+\delta_1\chi(M+1+z).
\eeaa
Hence
\begin{eqnarray*}
				&& d_1 \underline{U}'' + (c-\ep) \underline{U} ' + r_1 \underline{U} (1 - \underline{U} - a_2 \overline{V})\\
				& \geq & - \ep U' -   \delta_1 d_1  | \chi ''|   +  r_1 \underline{U} (1 - \underline{U} - a_2 \overline{V}) - r_1 U (1- U - a_2 V) \\
&=&- \ep U' -   \delta_1 d_1  | \chi ''|+r_1U\delta_1(1-a_2)\chi-r_1\delta_1\chi(1-U-a_2V+\delta_1\chi-a_2\delta_1\chi)\\
				& \geq &  \ep \eta  -  \delta_1 d_1  | \chi ''|     - r_1 \delta_1 (1 + \delta_1)\\
				& \geq & 0,
\end{eqnarray*}
using $\chi'\le 0$, $a_2<1$, $\chi\le 1$ and \eqref{delta12}.
Similarly, for $M < z < M+1$, we have 
\beaa
\underline{U}(z)=U(z)-\delta_1\chi(M+1-z)e^{-\lambda z},\; \overline{V}(z)=V(z)+\delta_2\chi(M+1-z),
\eeaa
and so
			\begin{eqnarray*}
				&& d_1 \underline{U}'' + (c-\ep) \underline{U} ' + r_1 \underline{U} (1 - \underline{U} - a_2 \overline{V})\\
				& \geq &  - \ep {U}' -   \delta_1 d_1 ( | \chi ''| + 2 \lambda |\chi '| + \lambda^2 ) 
+  r_1 \underline{U} (1 - \underline{U} - a_2 \overline{V}) - r_1 U (1- U - a_2 V) \\
				& \geq &  \ep \eta  -  \delta_1 d_1 ( | \chi ''| + 2 \lambda |\chi '| + \lambda^2 )     - r_1 \delta_1 (1 +  \delta_1)\\
				& \geq & 0.
			\end{eqnarray*}
One can also check that 
$$d_2 \overline{V} '' + (c-\ep) \overline{V} ' + r_2 \overline{V} (1 -\overline{V} - b_1 \underline{U} ) \leq 0$$
holds for $z\in(-M-1,-M)\cup(M,M+1)$.
This concludes the proof.
\end{proof}
Finally, the upper estimate on the spreading of the $u$-component of the homogeneous system~\eqref{Ph} follows from the next proposition. Since we will be able to use this upper-solution as is in the context of the homogeneous system~\eqref{Ph}, we will not go into the details of its construction.
	\begin{proposition}[\cite{gl19}, Proposition~1.5]\label{prop:gl_15}
		Let $s_1 \in (s_{1,GL}^*, 2 \sqrt{d_1 r_1})$. There exist $\delta^* >0$ and~$(s_1^\delta,s_2^\delta)_{\delta \in (0,\delta^*)}$ such that
				$$s_2^\delta \nearrow s_2^*, \qquad s_1^\delta \searrow s_1,$$ 
			as $\delta \to 0$, and system~\eqref{Ph} admits an upper-lower solution  $(\overline{u},\underline{v})$ satisfying the following properties.
		\begin{itemize}
			\item For all $t \geq 0$, the function $\overline{u} (\cdot,t)$ is equal to $1$ on the left of the moving point $s_1^\delta t$. Also, for all $x \in \mathbb{R}$,
					$$\overline{u} (x ,0) \geq \min \{ 1 ,  e^{-\mu x}\},$$
					where $\mu$ is the smaller root of
					$$d_1 \mu^2 - s_2^* \mu + r_1 + \lambda (s_2^* - s_1) = 0,$$
					in which $\lambda$ is the smaller root of
					$$d_1 \lambda^2 - s_1  \lambda + r_1 (1 -a_2)=0 .$$
			\item For all $t \geq 0$, the function $\underline{v} (\cdot,t)$ is compactly supported, its supremum is less than $1 -\delta$, and its support is including in the moving half-line $[s_1^\delta t, +\infty)$.
			\item The pair $(\overline{u},\underline{v})$ spreads in the following sense,
			$$\lim_{t \to +\infty} \left\{ \sup_{(s_1^\delta + \ep )t \leq x \leq (s_2^\delta - \ep) t} \Big[ | \overline{u} (x,t) | + |\underline{v} (x,t) - (1-2\delta)| \Big] \right\} = 0 ,$$
			for any $\ep \in \left(0, \frac{s_2^\delta - s_1^\delta}{2} \right)$.
		\end{itemize}
	\end{proposition}
	One may recognize that $\lambda$ and $\mu$ solve the same quadratic equations as in our proof of Theorem~\ref{thm:nlp}. This is no coincidence as nonlocal pulling may also happen here.  
	
	In fact, in~\cite[Proposition 1.5]{gl19}, some parts of the first two items of this proposition are only stated at time $t=0$. This slight change is necessary because, later on, we will compare solutions of~\eqref{P} with this upper-lower solution $(\overline{u}, \underline{v})$, and due to climate change, we will do this on a moving half-line. Anyway, one may check that the pair $(\overline{u}, \underline{v})$ does satisfy those properties for all nonnegative times, due to the fact that they only depend on some moving variables. If one would like to avoid the technical difficulties, the Figure~4.3 in~\cite{gl19} should be convincing enough (notice that, in their notation, $s_2^\delta$ and $s_1^\delta$ respectively become~$c_1^\delta$ and~$c_2^\delta$).\medskip

\subsection{The homogeneous two-species competition system: a Liouville type result}

An important ingredient of our arguments will be a Liouville type result, which is that~\eqref{Ph} admits a unique bounded entire in time solution whose first component has positive infimum. We state it in the strong-weak competition case where \eqref{eq:hypothesis} holds. Still, similar results hold in other cases, in particular with the weak-weak competition, up to replacing $(1,0)$ by the unique stable coexistence steady state, provided it exists.

\begin{proposition}\label{prop:entire}
 	Assume that $a_2<1<b_1$. Suppose that $(u,v)$ is an entire in time solution of \eqref{Ph} such that $\delta\le u\le 1$ and $0\le v\le 1$ for some constant $\delta\in(0,1)$.
	Then $(u,v)\equiv (1,0)$.
\end{proposition}

\begin{proof}
	The proof consists in finding a Lyapunov function. Consider
	\be\label{F}
	F(u,v):=(u-1-\ln u)+\sigma v,
	\ee
	where $\sigma$ is a positive constant to be determined.
	The directional derivative of $F$ along the vector field
	\beaa
	(f,g):=(r_1u(1-u-a_2v),r_2v(1-b_1u-v))
	\eeaa
	is computed as
	\beaa
	\mathcal{X}F(u,v):=\nabla F\cdot(f,g)=-r_1(1-u)^2+(a_2r_1+\sigma r_2)v(1-u)-\sigma r_2v^2-\sigma r_2(b_1-1)uv.
	\eeaa
	Then we have
	\beaa
	\mathcal{X}F(u,v)&=&-\left\{r_1-\frac{\sigma r_2}{4}\left(1+\frac{a_2r_1}{\sigma r_2}\right)^2\right\}(1-u)^2-\sigma r_2(b_1-1)uv\\
	&&-\sigma r_2\left\{v-\frac{1}{2}\left(1+\frac{a_2r_1}{\sigma r_2}\right)(1-u)\right\}^2\\
	&\le&-\left\{r_1-\frac{\sigma r_2}{4}\left(1+\frac{a_2r_1}{\sigma r_2}\right)^2\right\}(1-u)^2-\sigma r_2(b_1-1)\delta v,
	\eeaa
	using $u\ge\delta$.
	
	Next, we write
	\beaa
	r_1-\frac{\sigma r_2}{4}\left(1+\frac{a_2r_1}{\sigma r_2}\right)^2=\frac{1}{4\sigma r_2}h(\sigma ),\; \mbox{ where } h(\sigma ):=[4\sigma r_1r_2-(a_2r_1+\sigma r_2)^2].
	\eeaa
	We compute
	\beaa
	h(\sigma )=-r_2^2\sigma ^2+(4-2a_2)r_1r_2\sigma -a_2^2r_1^2,\quad h'(\sigma )=-2r_2^2\sigma +(4-2a_2)r_1r_2.
 	\eeaa
	Using $a_2<1$, the function~$h$ has a unique maximal point at $\sigma _0:=(2-a_2)r_1/r_2>0$ with
	\beaa
	h(\sigma _0)=r_1^2[(2-a_2)^2-a_2^2]>0.
	\eeaa
	By choosing the constant $\sigma =\sigma _0$ in \eqref{F}, we find that
	$$\mathcal{X} F(u,v) \leq - \frac{h(\sigma_0)}{4 \sigma_0 r_2} (1-u)^2 - \sigma_0 r_2 (b_1 - 1 ) \delta v.$$
	On the other hand, for the given $\delta>0$ there is a positive constant $\beta$ such that
	\beaa
	(1-u)^2\ge \beta(u-1-\ln u),\;\forall\, u\in[\delta,1].
	\eeaa
	We conclude that
\begin{equation*}
	\mathcal{X}F(u,v)\le -\kappa F(u,v),\; \mbox{ with } \kappa:=\min\{ r_2(b_1-1)\delta,\beta h(\sigma _0)/(4\sigma _0r_2)\}>0 .
\end{equation*}
	This shows that $F$ is a Lyapunov function of the ODE system given by~\eqref{Ph} without diffusion. Moreover, $F(u,v)=0$ if and only if $(u,v)=(1,0)$.
	
	The proposition now follows by a similar proof to that of \cite[Theorem 1.1]{gs21}. We omit the details here.
\end{proof}

\section{The two species competition system facing climate change : stronger and faster competitor} 
\setcounter{equation}{0}

In this section and the next, we finally deal with the two species competition problem with climate change. Here we start with the proof of Theorem~\ref{thm:1_strongfast}, that is the situation when 
$$a_2 < 1 < b_1, \qquad s_1^* > s_2^*,$$
which means that $u$ is the strong competitor, $v$ is weak, and also $u$ is the faster of the two species.

We recall that, in Section~\ref{sec:prelim_single1}, we have already proved statements~$(i)$ and~$(ii)$, as well as the second limit of~$(iii)$. It also followed from Corollary~\ref{thm:scalar_shifting1} that
	$$\lim_{t \to +\infty} \left\{\sup_{x\le (s - \ep)t} |v(x,t)| + \sup_{x\ge (s_2^* +\ep)t} |v (x,t)| \right\} = 0,$$ 
for any $\varepsilon \in (0,s)$. Therefore, the remainder of Theorem~\ref{thm:1_strongfast} follows from the next proposition.
\begin{proposition}\label{prop:u_strongfast1}
	Assume that $a_2 < 1 < b_1$ and $0< s < s_2^* < s_1^*$. Let $(u,v)$ be the solution of~\eqref{P} with initial data $(u_0,v_0)$ satisfying \eqref{i-bd}, and such that the support of $v_0$ is bounded from above. 
Then
$$\lim_{t \to +\infty}\left\{ \sup_{(s + \varepsilon)t \leq x \leq (s_1^* -\varepsilon)t} \Big[ |u(x,t)-1| + |v(x,t)| \Big] \right\} =0,$$
for any $\varepsilon \in \left(0, \frac{s_1^*-s}{2} \right)$.
\end{proposition}
\begin{proof}
	The proof proceeds in three steps. In the first one, we will show that the species $u$ persists in any moving frame with speed $c \in (s_2^*,s_1^*)$, due to $v$ going to extinction in such a moving frame. In the second step, using a nonmoving subsolution, we find that $u$ still persists in slower moving frames. Lastly, the convergence to the steady state $(1,0)$ then follows from the Liouville Proposition~\ref{prop:entire}.\medskip
	
	\paragraph{\textit{Step 1.}} Let us place ourselves in any moving frame with speed
$$c \in (s_2^* , s_1^*) .$$
Since~$v$ satisfies
$$v_t \leq d_2 v_{xx} + r_2 v ,$$
we infer by a parabolic comparison principle and the explicit solution for this linear equation that
$$\limsup_{t \to +\infty} \sup_{x \geq s_2^* t } v(x,t) =0.$$
We point out to the proof of~Theorem~\ref{thm:lbsf} above, where a similar argument was made.

Recalling that $c < s_1^* = 2 \sqrt{d_1 r_1}$, we may fix $\delta >0$ such that
$$r_1 (1 - 2\delta)  > \frac{c^2}{4 d_1}.$$
Then, take any $L_1$ large enough such that
\begin{equation}\label{eq:L1}
	 \frac{d_1 \pi^2}{4 L_1^2} < r_1 ( 1 - 2 \delta) - \frac{c^2}{4 d_1 } , 
\end{equation}
and, thanks to $c > s_2^* > s$, any $T$ large enough such that
$$\inf_{t \geq T, x \in (ct- L_1,ct+L_1) } \alpha_1(x-st) \geq 1 - \frac{\delta}{2} ,$$ 
$$\sup_{ t\geq T, x \in (ct -L_1, ct+L_1)} v (x,t) \leq \frac{\delta}{2 a_2}.$$
Then, on the one hand, we have, for all $t \geq T$ and $x \in (ct-L_1,ct+L_1)$, that
$$u_t \geq  d_1 u_{xx} + r_1 u (1- \delta - u). $$
On the other hand, on the same subdomain, the function
$$\underline{u}_1 (x,t) :=  \left\{ 
\begin{array}{ll}
	\eta_1 e^{-\frac{c}{2 d_1} (x-ct)} \cos \left( \frac{ \pi (x-ct) }{2L_1} \right) & \text{ if } x \in (ct - L_1 , ct + L_1), \vspace{3pt}\\
	0 & \text{ otherwise},
	\end{array}
	\right.$$
satisfies
\begin{eqnarray*}
	& & (\underline{u}_1)_t - d_1 (\underline{u}_1)_{xx} - r_1 \underline{u}_1 (1 -\delta - \underline{u}_1)\\
	& = & \left( \frac{c^2}{4d_1} - r_1 (1 - \delta) + d_1 \frac{\pi^2}{4 L_1^2} \right) \times \underline{u}_1 + r_1 \underline{u}_1^2 \\
	& \leq &  r_1 \underline{u}_1  \left(\underline{u}_1  - \delta \right) \\
	& \leq & 0 ,
\end{eqnarray*}
where the last inequality holds provided that $\eta_1$ is small enough. 

Due to $u$ being positive and $\underline{u}_1$ having compact support, we can decrease $\eta_1$ again if necessary so that $\underline{u}_1 (\cdot,T) \leq u(\cdot,T)$ and apply the comparison principle to conclude that
\begin{equation}\label{u_strongfast_step1}
\liminf_{t \to +\infty} \inf_{x \in (ct -L_1/2 ,ct +L_1/2)} u(x,t) > 0  .
\end{equation}
For later use, we highlight that according to~\eqref{eq:L1},~$L_1$ may be arbitrarily large here.\medskip

\paragraph{\textit{Step 2.}} Here we let $0 < \ep < \frac{s_1^* -s}{2}$ and prove that \begin{equation}\label{u_strongfast_step2}
	\liminf_{t \to +\infty}\inf_{(s+\varepsilon) t \leq x \leq (s_1^* - \varepsilon) t}  u(x,t) >0.
\end{equation}
We proceed by contradiction and assume that there exist sequences $t_n \to +\infty$ and $x_n \in ((s+ \varepsilon) t_n, (s_1^*-\varepsilon) t_n )$ such that
$$u (x_n , t_n) \to 0,$$
as $n \to +\infty$. 

First, we pick $c \in (s_1^* - \varepsilon, s_1^*)$ and define
$$t'_n = \frac{x_n}{c}.$$
According to the first step and more precisely \eqref{u_strongfast_step1}, we have, for any $L_1 >0$, that
\begin{equation}\label{u_strongfast_step1_bis}
	\liminf_{n \to +\infty} \inf_{x \in (-L_1/2, L_1 /2)} u(x+ x_n , t'_n ) > 0 .
	\end{equation}
Next, let us define
$$\underline{u}_2(x) :=  \left\{ 
\begin{array}{ll}
	\eta_2  \cos \left( \frac{ \pi x }{2L_2 } \right) & \text{ if } x \in (- L_2 , L_2), \vspace{3pt}\\
	0 & \text{ otherwise}.
\end{array}
\right.$$
It follows from~\eqref{u_strongfast_step1_bis} that, provided $\eta_2$ is small enough and $L_1 > L_2$, we have
$$u (x_n + \cdot , t'_n) \geq \underline{u}_2 (\cdot),$$
for all $n$ large enough.

Moreover, we will show that
$$ \underline{u}_2 (x - x_n )$$
is a subsolution for the $u$-equation on an appropriate subdomain. By similar computations as in the previous step, thanks to $a_2 <1$, one may first check that 
$$(\underline{u}_2)_t - d_1 (\underline{u}_2)_{xx} -r_1 \underline{u}_2  (1 - a_2 -  \delta - \underline{u}_2) \leq 0, $$
for some $\delta >0$, provided that $L_2$ and $\eta_2$ are respectively large and small enough. Then, notice that $t'_n < t_n$, and for any $t \in (t'_n , t_n)$ and $x \in (x_n - L_2, x_n + L_2)$, we have 
$$x - st \geq x_n - L_2 - s t_n  \geq \varepsilon t_n - L_2.$$
In particular, up to increasing $n$, we have that
$$\alpha_1(x-st) \geq 1 - \delta .$$
It follows from the above that, for all $t \in (t'_n, t_n)$ and $x \in (x_n - L_2, x_n+ L_2)$, we have
$$(\underline{u}_2)_t (x - x_n)  - d_1 (\underline{u}_2)_{xx} (x-x_n) -r_1 \underline{u}_2  (x-x_n) (\alpha_1 (x- s t)  - a_2 v(x,t) - \underline{u}_2 (x-x_n)) \leq 0.$$
We can now apply the comparison principle to find that
$$u (x_n, t_n ) \geq \underline{u}_2 (0) =\eta_2 >0 ,$$
for any $n$ large enough. We have reached a contradiction. \medskip
	
\paragraph{\textit{Step 3. Conclusion}} We are now in a position to end the proof of Proposition~\ref{prop:u_strongfast1}. This also follows from a contradiction argument, as in the proof of \cite[Theorem 1.1]{w22}, thanks here to Proposition~\ref{prop:entire}
(instead of \cite[Proposition 2.1]{w22}).

For the reader's convenience, we provide the details as follows.
Given $\ep\in(0,(s_1^{*}-s)/2)$. Suppose that there is a sequence $\{(x_n,t_n)\}$ with $t_n\to +\infty$ as $n\to +\infty$ and
\beaa
(s+\ep)t_n\le x_n\le (s_1^{*}-\ep)t_n,\;\forall\, n,
\eeaa
such that
\be\label{theta}
|u(x_n,t_n)-1|+v(x_n,t_n)\ge\theta,\;\forall\, n,
\ee
for some positive constant $\theta$.
Let $(u_n,v_n)(x,t):=(u,v)(x+x_n,t+t_n)$ for $x\in\bR$, $t>-t_n$.
Then, by the regularity theory of parabolic equations, up to a subsequence, the sequence $\{(u_n,v_n)\}$ converges to $(u_\infty,v_\infty)$
locally uniformly for $(x,t)\in\bR^2$ such that $(u_\infty,v_\infty)$ is an entire solution of~\eqref{Ph}.

Now, due to \eqref{u_strongfast_step2} in the previous step, one can infer that $$u_\infty (x,t) \ge  \delta ,$$ for some $\delta >0$, for all $(x,t)\in\bR^2$.
Putting this together with $u_\infty\le 1$ and $0\le v_\infty\le 1$, we can apply Proposition~\ref{prop:entire} to get that $(u_\infty,v_\infty)\equiv (1,0)$. However, this contradicts~\eqref{theta}, and concludes the proof of Proposition~\ref{prop:u_strongfast1}.\end{proof}


\section{The two species competition system facing climate change : stronger but slower competitor}

Next, we turn to the proofs of Theorem~\ref{thm:2_strongslow} and Proposition~\ref{prop:diff_outcomes}. Hence, we again assume that $a_2 < 1 < b_1$, but now 
$$s_1^* < s_2^* ,$$
so that $u$ is the stronger but slower competitor. 

Thanks to Corollaries~\ref{thm:scalar_shifting1} and~\ref{th:general} (where we recall that the roles of~$u$ and~$v$ can be inverted without loss of generality), we are already done with the case when $s \geq s_1^*$. Thus, we will focus here on the case when 
$$s < s_1^* .$$
We will start with some general lemmas and then divide this case into further subcases depending on whether $s > s_{1,GL}^*$, $s < s_{1,nlp}^*$ or $s \in (s_{1,nlp}^*, s_{1,GL}^*)$. For later use, we also recall here that we always have $s_{1,nlp}^* \leq s_1^*$; see Theorem~\ref{thm:nlp} above.

\subsection{A few lemmas}

We start by exhibiting some subdomain where the solution $(u,v)$ converges to the steady state $(0,1)$, i.e. the faster species $v$ thrives away from both its competitor $u$ and the climate change.

\begin{lemma}\label{subcase1_lemma1}
	Assume that $a_2 < 1 < b_1$ and $s < s_1^* < s_2^*$. Let $(u,v)$ be the solution of \eqref{P} with initial data $(u_0,v_0)$ satisfying \eqref{i-bd} such that the support of $u_0$ is bounded from above. Then
	$$\limsup_{t \to +\infty} \left\{ \sup_{x \geq (s_1^* + \ep)t} |u(x,t)| + \sup_{(s_1^* +\varepsilon)t \leq x \leq (s_2^{*}-\varepsilon)t } |v(x,t) -1|  \right\} = 0, $$
	for any $\ep \in \left( 0, \frac{s_2^* - s_1^*}{2} \right)$.
\end{lemma}
\begin{proof}
	The supremum limit for $u$ is already known from Corollary~\ref{thm:scalar_shifting1}, more precisely from statement~$(i)$. As far as the convergence of $v$ to $1$ is concerned, we only sketch the argument which proceeds similarly as the proof of Proposition~\ref{prop:u_strongfast1}. 
	
	First, for any $c \in (0, s_2^*)$, one defines
	$$\underline{v}_1 (x,t) := \left\{ 
	\begin{array}{ll}
		\eta_1 e^{-\frac{c}{2 d_2} (x-ct)} \cos \left( \frac{ \pi (x-ct) }{2L_1} \right) & \text{ if } x \in (ct - L_1 , ct + L_1), \vspace{3pt}\\
		0 & \text{ otherwise},
	\end{array}
	\right.$$
	which for small $\eta_1 >0$ and large $L_1 >0$ satisfies
	\begin{equation}\label{subcase1_subsol1}
		(\underline{v}_1)_t - d_2 (\underline{v}_1)_{xx} - r_2 \underline{v}_1 (1 - \delta -\underline{v}_1 ) \leq 0,
	\end{equation}
	for some $\delta>0$. If furthermore $c > s_1^*$, due to $u$ going to $0$ and $\alpha_2$ going to $1$ in the moving frame with speed~$c$, we get that $\underline{v}_1$ is also a subsolution of the $v$-equation in~\eqref{P}. By the comparison principle, eventually we find that
	$$\lim_{t \to +\infty} \inf_{x \in (ct - L_1/2, ct+ L_1/2)} v(x,t) >0,$$
	for any $c \in (s_1^*, s_2^*)$.
	
	Second, we find another subsolution of~\eqref{subcase1_subsol1} of the type
	$$\underline{v}_2 (x,t) := \left\{ 
	\begin{array}{ll}
		\eta_2  \cos \left( \frac{ \pi x }{2L_2} \right) & \text{ if } x \in ( - L_2 ,  L_2), \vspace{3pt}\\
		0 & \text{ otherwise},
	\end{array}
	\right.$$
	with $\eta_2,L_2 >0$. Unlike $\underline{v}_1$, and as in Step~2 of the proof of Proposition~\ref{prop:u_strongfast1}, this only gives a subsolution of the $v$-equation in~\eqref{P} when shifted in space and until $u$ and the climate change eventually catch up. Still, this is enough to eventually get that
	$$\lim_{t \to +\infty} \inf_{(s_1^* + \ep ) t \leq x \leq (s_2^* -\ep) t} v(x,t) >0,$$
	for any $\ep \in \left( 0, \frac{s_2^* -s_1^*}{2} \right)$. Applying a Liouville type result, according to which~$1$ is the only bounded entire in time solution with positive infimum of the scalar Fisher-KPP equation~\cite{bhn}, one may then conclude the proof of Lemma~\ref{subcase1_lemma1}.
\end{proof}

We immediately improve this, by pointing out that the presence of $v$ should in fact slow down the speed at which~$u$ spreads and replaces~$v$.
\begin{lemma}\label{lem:lessthanGL}
	Assume that $a_2 < 1 < b_1$ and $s < s_1^* < s_2^*$. Let $(u,v)$ be the solution of \eqref{P} with initial data $(u_0,v_0)$ satisfying \eqref{i-bd} such that the support of $u_0$ and $v_0$ are both bounded from above. Then
\begin{equation}\label{eq:claim_smtf_upper}
	\lim_{t \to +\infty} \left\{  \sup_{ x \geq ( \overline{s} + \ep ) t  }  |u(x,t)  |  +  \sup_{ ( \overline{s} + \ep ) t \le  x \le (s_2^*  - \ep)t} | v (x,t) -1|  \right\} = 0,
\end{equation}
	for any $\ep \in \left( 0, \frac{s_2^* - \overline{s}}{2} \right)$, where
	$$\overline{s} = \max\{ s_{1,GL}^*,s\}.$$
\end{lemma}
Notice that, together with the results of Section~\ref{sec:prelim_single1}, this ends the proof of statement~$(iii)$ of Theorem~\ref{thm:2_strongslow} in the case when $s > s_{1,GL}^*$.

\begin{proof}
Recall that $s_{1,GL}^*$ was characterized in Section~\ref{sec:prelim_gl}, and in particular in Theorem~\ref{thm:gl_add}. The proof of Lemma~\ref{lem:lessthanGL} is the same regardless of whether $s_{1,GL}^* = s_{1,mtf}^*$ or $s_{1,GL}^* = s_{1,nlp}^*$, or more accurately, we do not need to make the distinction here as we rely on the general upper-lower solution from Proposition~\ref{prop:gl_15}.

We start the proof of~\eqref{eq:claim_smtf_upper} by considering another perturbed homogeneous system, i.e.
\be\label{p-}
\bss
u_t=d_1u_{xx}+r_1 u(1   - u- a_2 v),\;x\in\bR,\, t>0, \vspace{3pt}\\
v_t=d_2v_{xx}+r_2v(1 -  \delta -  b_1 u-v),\;x\in\bR,\, t>0 ,
\ess
\ee 
with $\delta >0$, whose solution we will denote by $(u_{-\delta}, v_{-\delta})$. Again, this system is equivalent to~\eqref{Ph} up to some renormalization and all of Section~\ref{sec:prelim_gl} applies up to some change of notation. In particular, the component $u_{-\delta}$ spreads with some speed $s_{1,GL,-\delta}^*$ in a way analogous to Theorem~\ref{thm:gl},  and $s_{1,GL,-\delta}^* \to s_{1,GL}^*$ as $\delta \to 0$. Furthermore, by Proposition~\ref{prop:gl_15}, system~\eqref{p-} admits an upper-lower solution $(\overline{u}_\delta, \underline{v}_\delta)$, whose first component, among other properties, spreads with some speed $s_1^\delta$ which up to reducing $\delta >0$ can be made strictly larger but arbitrarily close to $\overline{s} = \max \{ s_{1,GL}^* , s \}$.

Due to $\alpha_2 (+\infty) = 1$, there exists $T>0$ large enough so that the solution~$(u,v)$ of \eqref{P} is a lower-upper solution of \eqref{p-} on the parabolic subdomain 
$$Q_T := \{ t \geq T, \ x \geq (s+ \ep) t \}.$$
Our goal is now to find some spatial shift $X$ such that the inequalities
\begin{equation}\label{eq:almost}
	u (x, t) \leq \overline{u}_\delta (x - X, t -T), \quad v(x,t) \geq \underline{v}_\delta (x-X, t-T),
\end{equation}
hold on the parabolic boundary of~$Q_T$, that is on $\{ t = T, \ x \geq (s+ \ep ) T\}$ and $\{ t \geq T, \ x = (s+ \ep) t\}$, and then by the comparison principle they will hold on the whole $Q_T$.

On the one hand, recall from Proposition~\ref{prop:gl_15} that 
$$\overline{u}_\delta = 1, \quad \underline{v}_\delta = 0$$
for $x \leq s_1^\delta t$. Then, provided $\delta$ and $\ep$ are small enough so that $s_1^\delta > s+ \ep$, we have that
\begin{equation}\label{eq:X_cond1}
	X > s_1^\delta T
\end{equation}  
implies that
$$\overline{u}_\delta ((s+\ep) t - X, t - T) = 1 \geq u ((s+\ep)t,t),$$
$$\underline{v}_\delta ((s+\ep)  t - X, t-T) = 0 \leq v ((s+\ep)t,t),$$
and in particular \eqref{eq:almost} holds for $t \geq T$ and $x = (s+ \varepsilon) t$.

On the other hand, it is more intricate to check~\eqref{eq:almost} at time~$T$. First, we have from Proposition~\ref{prop:gl_15} that
	$$\overline{u}_\delta (x ,0) \geq \min \{ 1 ,  e^{-\mu_\delta x}\},$$
	where $\mu_\delta$ can be made arbitrarily close to $\mu$ the smaller (possibly double) root of
	$$d_1 \mu^2 - s_2^* \mu + r_1 + \lambda (s_2^* - \overline{s}) = 0,$$
and $\lambda$ is the smaller (possibly double) root of
$$d_1 \lambda^2 -  \overline{s} \lambda + r_1 (1 -a_2)=0 .$$
Notice that both $\lambda$ and $\mu$ are well-defined due to $\overline{s} \geq s_{1,GL}^* \geq s_{1,nlp}^*$; see our proof of Theorem~\ref{thm:nlp}.

Then, for all $t \geq 0$ and~$x \in \mathbb{R}$, one has (possibly up to some shift but without loss of generality) that
	\begin{equation}\label{eq:last?_super}
		u (x,t) \leq \min \{1 , e^{-\lambda (s_2^* -s_1)t} \times e^{-\mu_\delta ( x - s_2^* t)} \},
	\end{equation}
	where $s_1 \in [ \overline{s} ,s_2^*) $ is such that
	$$d_1 \mu_\delta^2 - s_2^* \mu_\delta + r_1 + \lambda (s_2^* - s_1) \leq 0.$$
 Inequality~\eqref{eq:last?_super} comes from the fact that its right-hand term is a supersolution of the $u$-equation of~\eqref{P} with $\alpha_1\equiv 1$ and $v \equiv 0$.
 
\begin{eqnarray*}
	u (\cdot, T) &  \leq & \min \{1, e^{-\lambda (s_2^* -s_1) T} \times e^{-\mu_\delta (x-s_2^* T)} \} \\
	& \leq & \min \{ 1, e^{-\mu_\delta (x - X)} \}\\
	& \leq & \overline{u}_\delta (\cdot - X, 0),
\end{eqnarray*}
provided that
\begin{equation}\label{eq:X_cond2}
	\mu_\delta X \geq \mu_\delta s_2^* T - \lambda (s_2^* - s_1) T.
\end{equation} 
However, if we take $X$ too large, then as $\underline{v}$ is shifted too much to the right it cannot be true that $ v(\cdot, T) \geq \underline{v} (\cdot -  X, 0)$. Still, let us recall by Lemma~\ref{subcase1_lemma1} that, up to increasing $T$,
$$\inf_{(s_1^* + \ep) T \leq x \leq (s_2^* - \ep) T } v (x,T) \geq 1 -\delta,$$
for any $\ep \in \left( 0 , \frac{s_2^* - s_1^*}{2} \right)$. Besides, the function $\underline{v}_\delta (\cdot,0)$ has a supremum less than $1-\delta$, and a compact support, say included in $(0,L)$ for some large enough $L>0$. It follows that the second inequality in~\eqref{eq:almost} holds for $t= T$ provided that
\begin{equation}\label{eq:X_cond3}
	( X, L  +X  )  \subset ( (s_1^* + \ep) T , ( s_2^* - \ep)T ) .
\end{equation}
Finally, one finds that $X$ can satisfy~\eqref{eq:X_cond1}, \eqref{eq:X_cond2} and~\eqref{eq:X_cond3} simultaneously, provided that~$T$ is large enough. Then one applies the comparison principle and conclude that	
$$\lim_{t \to +\infty} \sup_{x \geq (s_1^\delta + \ep )t }  |u (x,t) | = 0,$$
$$\liminf_{t \to +\infty} \inf_{(s_1^\delta + \ep )t \leq x \leq (s_2^\delta - \ep) t}   v(x,t)  \geq 1 - 2 \delta ,$$
for any $\ep \in \left(0, \frac{s_2^\delta - s_1^\delta}{2} \right)$. Since $s_1^\delta, s_2^\delta$ go respectively to $\overline{s}$, $s_2^*$ as $\delta \to 0$, this ends the proof.
\end{proof}

\subsection{Subcase 1: $s < s_{1,nlp}^*$}\label{sub:subcase1}

Now we turn to the (sub)case when $s < s_{1,nlp}^*$, i.e. the climate change is slower than the speed of~$u$ in the sense of some appropriate linearization around~$0$ taking count of the presence of $v$. In this subcase, we expect some hair-trigger effect to occur and for $u$ to spread regardless of the specifics of the initial data.

Recall that, thanks again to Corollaries~\ref{thm:scalar_shifting1} and~\ref{th:general}, we already know that
$$\lim_{t \to +\infty} \left\{\sup_{x\le (s - \ep)t} | v(x,t)| + \sup_{x\le (s -\ep)t}  | u(x,t)| + \sup_{x\ge (s_2^* +\ep)t} |v (x,t)| \right\} = 0,$$ 
for any $\ep >0$, and thanks to Lemma~\ref{lem:lessthanGL} together with $s < s_{1,nlp}^* \leq s_{1,GL}^*$, that
$$\lim_{t \to +\infty} \left\{ \sup_{x \geq (s_{1,GL}^* + \ep)t } |u(x,t)| + \sup_{(s_{1,GL}^* + \ep)t \le x \le ( s_2^* - \ep ) t} |v(x,t) -1| \right\} = 0,$$ 
for any $\ep \in \left(0 , \frac{s_2^* - s_{1,GL}^* }{2} \right)$.

It remains to prove that
$$\lim_{t \to +\infty} \left\{ \sup_{(s + \ep)t \le x \le ( s_{1,GL}^* - \ep ) t}  \Big[ |u(x,t)-1| + |v(x,t)| \Big]   \right\} = 0,$$ 
for any $\ep \in \left(0 , \frac{s_{1,GL}^* - s}{2} \right)$.\medskip

We start with the following lemma which shows that~$u$ at least spreads with the aforementioned linear speed $s_{1,nlp}^*$. A future step will be dedicated to showing that, once $u$ has grown enough behind the moving frame with speed $s_{1,nlp}^*$, it will be able to reach its full (nonlinear) speed $s_{1,GL}^*$.

\begin{lemma}\label{subcase1_lemma2}
	Assume that $a_2 < 1 < b_1$ and $s < s_{1,nlp}^* \leq s_1^* < s_2^*$. Let $(u,v)$ be the solution of \eqref{P} with initial data $(u_0,v_0)$ satisfying \eqref{i-bd} such that the support of $v_0$ is bounded from above. Then
$$\limsup_{t \to +\infty} \left\{  \sup_{(s+\varepsilon)t \leq x \leq (s_{1,nlp}^{*}-\varepsilon)t } \Big[ | u(x,t) -1| + |v(x,t)| \Big] \right\} = 0, $$
for any $\ep \in \left(0, \frac{s_{1,nlp}^* - s}{2}\right)$. 
\end{lemma}
\begin{proof}Recall Theorem~\ref{thm:nlp} and that either $s_{1,nlp}^* = 2 \sqrt{d_1 r_1 (1-a_2)}$ or $s_{1,nlp}^* > 2 \sqrt{d_1 r_1 (1-a_2)}$. Let us deal with those two situations separately. \smallskip
	
	Assume first that $s_{1,nlp}^* = 2 \sqrt{d_1 r_1 (1-a_2)}$, and let $\delta$ be an arbitrarily small positive constant. Due to $\alpha_1 (+\infty) =1$ and $v \leq 1$, up to some spatial shift and without loss of generality, we have for all $x \geq (s+\ep) t$ that
	$$u_t \geq d_1 u_{xx} + r_1 u (1 - \delta -a_2 -u).$$
	The spreading speed of solutions of this homogeneous scalar equation is $2 \sqrt{d_1 r_1 (1-\delta -a_2)}$, thus we can find a subsolution moving with any speed $c \in \left( s, 2 \sqrt{d_1 r_1 (1-\delta - a_2)} \right)$, of the form
	$$\underline{u}_1 (x,t) :=  \left\{ 
	\begin{array}{ll}
		\eta_1 e^{-\frac{c}{2 d_1} (x-ct)} \sin \left( \frac{ \pi (x-ct) }{L_1} \right) & \text{ if } x \in (ct , ct + L_1), \vspace{3pt}\\
		0 & \text{ otherwise},
	\end{array}
	\right.$$
	i.e. such that 
	$$(\underline{u}_1)_t \leq d_1 (\underline{u}_1)_{xx} + r_1 \underline{u}_1 (1 - \delta - a_2 - \underline{u}_1).$$
	By comparison, we find that
	$$\liminf_{t \to +\infty} \inf_{ x \in (ct + L_1/4  , ct + 3L_1/4)} u (x, t) >0.$$
	Similarly as the proof of Proposition~\ref{prop:u_strongfast1}, this can be improved to 
		$$\liminf_{t \to +\infty} \inf_{ (s+\ep) t \leq x \leq (2 \sqrt{d_1 r_1 (1-\delta - a_2)} - \ep) t} u (x, t) >0,$$
	for any $\ep \in \left(0, \frac{2 \sqrt{d_1 r_1 (1-\delta-a_2)}  - s}{2}\right)$, and then 
	$$\limsup_{t \to +\infty} \left\{  \sup_{(s+\varepsilon)t \leq x \leq (s_{1,nlp}^{*}-\varepsilon)t } \big[ | u(x,t) -1| + |v(x,t)| \big] \right\} = 0, $$
	by Proposition~\ref{prop:entire} and letting $\delta \to 0$.\smallskip
	
	Next, we turn to the case when $s_{1,nlp}^* > 2 \sqrt{d_1 r_1 (1-a_2)}$. Recalling that $s_2^* > s_1^*$, this means we are in case~$(ii)$ of Theorem~\ref{thm:nlp}, in which the construction of subsolutions was slightly more intricate.
	
	Due to $\alpha_1 (+\infty) =1$, $v \leq 1$ and 
	$$\lim_{t \to +\infty} \sup_{x \geq (s_2^* + \ep)t} | v (x,t)| = 0,$$
	for any $\delta , \ep >0$, there is $T >0$ such that, for all $x \geq (s+ \ep) t$, it holds
	$$u_t \geq d_1 u_{xx} + r_1 u (1 -\delta - u -a_2 H ((s_2^* +\delta) t -x) - a_2 {\delta} H (x- (s_2^* + \delta)t)  ).$$
	This leads us to consider
	\begin{equation}\label{eq:shifting1_delta}
	U_t = d_1 U_{xx} + r_1 U (1 -\delta - U -a_2 H ((s_2^* +\delta) t -x) - a_2 {\delta} H (x- (s_2^* + \delta)t)  ),
	\end{equation}
	which is of the same type as~\eqref{eq:shifting1}, so that Theorem~\ref{thm:nlp} applies up to some change of notation. We also refer to~\cite[Theorem 6]{ly22} for the formula in the general case of a step-like growth rate function, that one may also apply here.
	
	In particular, for the part that is of relevance to us here, if 
		\begin{equation}\label{eq:delta1}
			s_2^* + \delta \in \left( 2 \sqrt{d_1 r_1 (1-\delta - a_2 \delta)}, 2 \sqrt{d_1 r_1 (1 -\delta - a_2)}
+ 2\sqrt{d_1 r_1 a_2 (1-\delta)} \right),
		\end{equation}
	then the solutions of \eqref{eq:shifting1_delta} with compactly supported initial data spread with speed
	$$s_{1,nlp,\delta}^{*}  := s_2^* + \delta - \frac{ \frac{(s_2^* + \delta )^2}{4} - d_1 r_1 (1-\delta - a_2 \delta)}{\frac{s_2^* + \delta}{2} -  \sqrt{d_1 r_1 a_2 (1 -\delta)}}.$$
	{Notice that everything depends continuously on $\delta$ here. In particular, since we are in case~$(ii)$ of Theorem~\ref{thm:nlp} now, we may assume that $\delta>0$ is small enough so that~\eqref{eq:delta1} holds.}
	
{Then, according to the proof of Theorem~\ref{thm:nlp}, whose sketch we included in Section~\ref{sec:prelim_nlp}}, for any $c \in (s+ \ep, s_{1,nlp,\delta}^*)$, there exists a nonnegative subsolution~$\underline{U}$ of~\eqref{eq:shifting1_delta} whose support is included in 
	$$(s+ \ep) t \leq x \leq  (s_2^* + 2 \ep) t,$$
	whose $L^\infty$-norm can be made arbitrarily small, and such that
	$$\liminf_{t \to +\infty} \underline{U} (ct, t) > 0. $$
	We can then apply the comparison principle and letting $\delta \to 0$, we get that
	$$\liminf_{ t\to +\infty} u (ct,t) > 0,$$
	for any $c \in \left( s + \ep , s_{1,nlp}^* \right)$.
	The end of the argument proceeds as the proof of Proposition~\ref{prop:u_strongfast1}, so we omit the details.
\end{proof}
With Lemmas~\ref{subcase1_lemma1} and~\ref{subcase1_lemma2} in hand, we now know that there are two zones ahead of the climate change, in each of which one species persists. The remaining question is: where is the separation between those two zones, i.e. where does  the transition between~$(1,0)$ and~$(0,1)$ occur? 

We already know from Lemma~\ref{lem:lessthanGL} that this transition must occur behind the moving frame with speed $s^*_{1,GL}$ corresponding to the homogeneous system. We claim in Theorem~\ref{thm:2_strongslow} that the transition actually moves (asymptotically in time) at speed~$s^*_{1,GL}$. The heuristics here is that, due to $s_{1,GL}^* \geq s_{1,nlp}^* > s$, the parameter functions~$\alpha_i$ are close to constants in the moving frame of the invading front of~$u$ and beyond. One may then expect that solutions of~\eqref{P} and~\eqref{Ph} behave and spread similarly. While this is not so simple to check, Section~\ref{sec:prelim_gl} provides us with the necessary toolbox to reach the wanted conclusion.  \smallskip

Let us now address the remaining wanted property that $u$ spreads at least with speed $s_{1,GL}^*$. Recall that, by Theorem~\ref{thm:gl_add}, we have $s_{1,GL}^* = \max \{ s_{1,mtf}^*, s_{1,nlp}^* \}$, where $s_{1,mtf}^*$ denotes the minimal traveling front speed of~\eqref{Ph} connecting $(1,0)$ and $(0,1)$. Moreover, by Lemma~\ref{subcase1_lemma2}, we already know that $u$ spreads at least with speed $s_{1,nlp}^*$. Therefore, it only remains to consider the case when $s_{1,GL}^* = s_{1,mtf}^* > s_{1,nlp}^*$, and to prove that
\begin{equation}\label{eq:claim_smtf}
\lim_{t \to +\infty} \left\{  \sup_{ (s + \ep ) t \le  x\le (s_{1,mtf}^* - \ep)t} \Big[ |u(x,t) - 1 | +  | v (x,t)| \Big] \right\} = 0,
\end{equation}
for any $\ep\in (0, (s_{1,mtf}^* - s)/2 )$. To do this, we consider the perturbed homogeneous problem
\be\label{p+}
\bss
u_t=d_1u_{xx}+r_1 u(1-  \delta  - u- a_2 v),\;x\in\bR,\, t>0,\vspace{3pt}\\
v_t=d_2v_{xx}+r_2v (1 - b_1  u-v),\;x\in\bR,\, t>0,
\ess
\ee
where $\delta >0$, whose solutions we will denote by~$(u_\delta, v_\delta)$. Though this is not exactly of the same form as~\eqref{Ph}, one can pass from one system to another by some renormalizations. In particular, \eqref{p+} admits traveling front solutions connecting~$(1-\delta,0)$ and~$(0, 1)$ with {a minimal speed~$s_{1,mtf,\delta}^*$, and} the first component of $(u_\delta, v_\delta)$ spreads with some speed $s_{1,GL,\delta}^*$ in an analogous way to Theorem~\ref{thm:gl}. 
{Furthermore, the speeds $s^*_{1,mtf,\delta}, s^*_{1,GL,\pm \delta}$ converge respectively to~$s_{1,mtf}^*, s_{1,GL}^*$ as $\delta \to 0$}.

For the same reason, Proposition~\ref{prop:gl_subsol} applies and provides lower-upper solutions to~\eqref{p+}. More precisely, for any small $\ep >0$, there exist a lower-upper solution to~\eqref{p+} of the form
$$(\underline{U}_\delta, \overline{V}_\delta) (x - (s_{1,mtf,\delta}^* - \ep) t ),$$
which one may check satisfies
$$\sup \underline{U}_\delta < 1, \quad  \inf \overline{V}_\delta  >0 ,$$
and, for some $Z>0$ and all $z \geq Z$,
$$\underline{U}_\delta (z) \leq 0, \qquad \overline{V}_\delta (z) \geq 1 .$$
Next, due to $\alpha_1 ( + \infty) = 1$ and Lemma~\ref{subcase1_lemma2}, there exists $T>0$ such that
$$\forall x \geq (s + \ep) t, \quad \alpha_1 (x - st) \geq 1 - \delta,$$
and
$$\forall t \geq T , \quad u ((s+ \ep) t,t) \geq \sup \underline{U}_\delta \ \mbox{ and } \ v((s+\ep)t,t) \leq \sup \overline{V}_\delta. $$
It follows that the solution $(u,v)$ of~\eqref{P} is an upper-lower solution of \eqref{p+} on the parabolic subdomain $\{ t \geq T, \, x \geq (s+ \ep) t\}$, and that
$$u (x,t) \geq \underline{U}_\delta (x - (s_{1,mtf,\delta}^* - \ep) t - (s+ \ep) T + Z),$$
$$v (x,t) \leq \overline{V}_\delta (x - (s_{1,mtf,\delta}^* - \ep) t -  (s+ \ep) T +Z )$$
for any $(x,t)$ on the parabolic boundary on that same subdomain. Therefore we can apply the comparison principle and find that the same inequalities hold for all $t \geq T$ and~$x \geq (s+ \ep) t$. We then conclude that
$$\lim_{t\to+\infty} \inf_{ (s + \ep)t \leq x \leq (s_{1,mtf,\delta}^* - 2 \ep) t } u (x,t) \geq 1 - \delta_1,$$
where $\delta_1$ comes from Proposition~\ref{prop:gl_subsol}. Applying the Liouville type result, letting $\delta \to 0$ and recalling that $\ep >0$ is arbitrarily small, this ends the proof of~\eqref{eq:claim_smtf}, thus of Theorem~\ref{thm:2_strongslow}.

	
\subsection{Subcase 2: $s \in \left( s_{1,nlp}^* , s_{1,GL}^*\right) $}

We now consider the trickier case when $s \in (s_{1,nlp}^{*}, s_{1,GL}^*)$ and prove Proposition~\ref{prop:diff_outcomes}. That is, we verify our claim that, in this range, the large time behavior changes depending on $\alpha_1, \alpha_2$ and the initial data $u_0,v_0$.\medskip

 We first construct a situation where $u$ goes to $0$, at least in the sense of the first item of Proposition~\ref{prop:diff_outcomes}. To do this, one may typically choose 
 $$\alpha_2 \equiv  \alpha_1 (\cdot + M),$$
	for some large positive $M$. Then $\alpha_2$ is ``much better'' than $\alpha_1$, i.e. there will be a large interval where $\alpha_2$ is close to $1$ while $\alpha_1$ is still negative. This means that $v$ may thrive in the moving frame with speed $s$. 
	 Due to $s > s_{1,nlp}^{*}$, i.e. the unfavourable zone moves faster than the spreading of $u$ in ``presence'' of $v$, this can prevent~$u$ from benefiting from its competitive advantage.
	
		Let us make this more rigorous. We construct an upper-lower solution of the following type:
		$$ \overline{u} (x,t) =\delta  \times \left\{
		\begin{array}{ll}
			1 & \ \mbox{ if } x \leq s t, \\
			e^{-\lambda (x -st) } & \ \mbox{ if } st < x \leq (s_2^* -  \sqrt{\delta})  t,\\
			 e^{-\lambda (s_2^* - \sqrt{\delta}- s) t} e^{-\mu (x- (s_2^* -\sqrt{\delta}) t)}& \ \mbox{ if } x > (s_2^* - \sqrt{\delta} ) t ,
		\end{array}
		\right.
		$$
		where $\delta, \lambda, \mu > 0$ to be specified later, and
			$$ \underline{v} (x,t) = \left\{
		\begin{array}{ll}
			0 & \ \mbox{ if } x \leq s t - L_1, \\
			\psi_1 (x -st )  & \ \mbox{ if } st - L_1 < x \leq s t,\\
			1 - 2 (b_1 + 1) \delta & \ \mbox{ if } s t < x \leq (s_2^* - \sqrt{\delta} ) t \\
			\psi_2 (x -(s_2^* - \sqrt{\delta}) t) & \ \mbox{ if } (s_2^* - \sqrt{\delta} ) t < x \leq (s_2^* - \sqrt{\delta}) t + L_2, \\
			0 & \ \mbox{ if } x > (s_2^* - \sqrt{\delta}) t + L_2
		\end{array}
		\right.
		$$
		where $\psi_1, \psi_2$ are positive parts of compactly supported subsolutions of the $v$-equation with $\alpha_2 \geq 1 - \delta$ and $u \leq \delta$. That is,
				$$d_2 \psi_1  '' + s \psi_1 ' + r_2 \psi_1 (1- \delta - \psi_1 - b_1 \delta) = 0\quad\mbox{in $(-\infty,0)$,}$$
	 			$$ d_2  \psi_2 '' + (s_2^* - \sqrt{\delta} ) \psi_2 ' + r_2 \psi_2 (1- \delta - \psi_2 - b_1 \delta) = 0\quad\mbox{in $(0,\infty)$,}$$
	 	together with
	 		$$\psi_1 ' (0) = \psi_2 ' (0) = 0,$$
	 		$$ \psi_1 (0)  = \psi_2 (0) = 1 - 2 (b_1 + 1 ) \delta > 0, $$
	 	where the last inequality holds provided that $\delta < \frac{1}{2 (b_1 + 1)}$, and
	 		$$  -L_1 = \sup \{ x <0  \, | \ \psi_1 (x) = 0 \} < 0 ,$$
	 		$$ L_2 = \inf \{ x >0 \, | \ \psi_2 (x) = 0 \} >0 .$$
	 	This choice already ensures that the function~$\underline{v}$ is continuous. Both $L_1$ and $L_2$ are well-defined real numbers by a standard ODE analysis (see also~\cite{aw75}), provided that
	 	$$s^2 < 4 d_2 r_2 (1- \delta - b_1 \delta),$$
	 	$$(s_2^* - \sqrt{\delta})^2 < 4 d_2 r_2 (1-\delta - b_1 \delta) .$$
	 	Indeed, these inequalities make $(\psi = 0, \psi ' = 0)$ a spiral point in the respective ODEs satisfied by $\psi_1$ and $\psi_2$. Moreover, thanks to $s < s_{1,GL}^* < s_2^* = 2 \sqrt{d_2 r_2}$, both {inequalities} are satisfied if $\delta$ is small enough.

	 	Let us now check that the pair $(\overline{u},\underline{v})$ is an upper-lower solution of~\eqref{P} when $\alpha_1, \alpha_2$ and the parameters $\delta, \lambda, \mu$ are suitably chosen. Let us first deal with the first differential inequality
$${\overline{u}_t \geq  d_1 \overline{u}_{xx} + r_1 \overline{u} ( \alpha_1 - \overline{u} - a_2  \underline{v} ).}$$
Provided that 
	 	\begin{equation*}\label{choice_alpha1}
	 		\alpha_1 < 0 \quad \text{ on } \  (-\infty, 0] ,
	 	\end{equation*}
	 	we have for $x < st$ that
	 		$$\overline{u}_t = 0 \geq   {d_1}\overline{u}_{xx} + r_1 \overline{u} ( \alpha_1 - \overline{u} - a_2  \underline{v} ) .$$
	 	For $x > st$, first notice that 
		\begin{equation}\label{v_major}
			\underline{v} (x,t) \geq (1- 2 (b_1 +1) \delta) \times H( (s_2^* - \sqrt{\delta}) t -x).
		\end{equation}
	Due to $s > s_{1,nlp}^*$, provided that $\delta$ is small enough, we may pick $0< \lambda \leq \mu $ as in our proof of Theorem~\ref{thm:nlp} so that
	 		$$\overline{u}_t \geq  {d_1}\overline{u}_{xx} + r_1 \overline{u} \left[ 1 - a_2  (1- 2\delta - 2 b_1 \delta ) H( (s_2^* - \sqrt{\delta}) t -x)  \right] .$$
	 	Together with $\alpha_1 \leq 1$ and~\eqref{v_major}, this implies the wanted differential inequality. Notice also that the left spatial derivative of $\overline{u}$ is larger than its right spatial derivative at both points $x = st$ and $x = ( s_2^* - \sqrt{\delta}) t$.
	 		
	 	Next, we turn to the second differential inequality, i.e. we check that
	 			$$ \underline{v}_t \leq {d_2}\underline{v}_{xx} + r_2 \underline{v} (\alpha_2 - \underline{v} - b_1 \overline{u}).$$
	 	This is trivially satisfied for $x < st - L_1$ and $x > (s_2^* - \sqrt{\delta}) t + L_2$. Provided that 
	 	\begin{equation*}\label{choice_alpha11}
	 		\alpha_2 > 1 - \delta \quad \text{ on } [-L_1, +\infty),
	 	\end{equation*} 
		the same differential inequality is satisfied in all cases $st - L_1 < x < st$, $st < x  < (s_2^* - \sqrt{\delta})t$ and $(s_2^* -  \sqrt{\delta}) t < x < (s_2^* - \sqrt{\delta}) t + L_2$. This comes from the definition of $\psi_1, \psi_2$ and the fact that~$\overline{u}\leq \delta$. Finally, the left spatial derivative of $\underline{v}$ is everywhere less than or equal to its right derivative.
		
		To sum up, what we have shown so far is that, for small enough $\delta >0$, there exist $\alpha_1, \alpha_2$ and $\lambda, \mu >0$ so that $(\overline{u},\underline{v})$ as defined above is an upper-lower solution of~\eqref{P}. In particular, if the initial data and this upper-lower solution are correctly ordered, we immediately infer that the $u$-component converges to $0$ uniformly as $t \to +\infty$ on the moving set $((s+\varepsilon)t,+\infty)$, for any $\varepsilon >0$. On the other hand, at this stage we only know that
		$$\liminf_{t \to +\infty} \inf_{(s+\ep)t \leq t \leq (s_2^*- \sqrt{\delta}) t} v(x,t) \geq 1- 2 (b_1 + 1) \delta.$$
		However, we cannot pass to the limit as $\delta \to 0$ in this inequality, as $\delta$ was picked before the initial data and even the functions $\alpha_1, \alpha_2$. Still, due to $s < s_2^*$ and the convergence of $u$ to $0$ on the moving set $((s +\varepsilon)t, +\infty)$, the conclusion that 
					$$\liminf_{t \to +\infty} \inf_{(s+\ep)t \leq t \leq (s_2^*-\ep )t} v(x,t) = 1,$$
					may now follow by Theorem~\ref{thm:lbsf} (see also Corollary~\ref{th:general}).\medskip

	
	Let us now construct a situation where
		$$\lim_{t \to \infty}\left\{\sup_{(s+\varepsilon) t \leq x \leq (s_{1,GL}^* - \varepsilon) t} \Big[ | u(x,t) - 1| + |v(x,t)| \Big] \right\} = 0,$$
	$$\lim_{t \to \infty}\left\{\sup_{(s_{1,GL}^* +\varepsilon) t \leq x \leq (s_{2}^* - \varepsilon) t} \Big[ | u(x,t) | + |v(x,t) -1| \Big] \right\} = 0,$$
	for any $\ep \in \left(0 , \min \left\{\frac{s_{1,GL}^* - s}{2}, \frac{s_2^* - s_{1,GL}^* }{2} \right\} \right)$. The second limit is already known thanks to Lemma~\ref{lem:lessthanGL}, so we may focus on the first one.
	
	This in fact proceeds similarly to the first subcase, dealt with in Section~\ref{sub:subcase1}. That is, we will again compare the solution with $(\underline{U}_\delta, \overline{V}_\delta)$ a lower-upper solution of the perturbed problem~\eqref{p+} provided by Proposition~\ref{prop:gl_subsol}. Recall that this lower-upper solution moves with some speed $s_{1,mtf,\delta}^*$, which as $\delta \to 0$ approaches $s_{1,mtf}^*$, which is none other than $s_{1,GL}^*$ due to $s_{1,GL}^* > s_{1,nlp}^*$ here.
	
	Beforehand, let us check that, for any $\eta >0$, there exists $\alpha_1$, $\alpha_2$ and an initial datum $(u_0,v_0)$ such that
	\begin{equation}\label{eqn:x_delta}
		 \forall x \geq x_\eta, \quad \alpha_1 (x) \geq 1 - \eta, \quad \mbox{and} \quad \forall t \geq 0 , \quad u( st + x_\eta ,t) \geq 1 - \eta  .
		\end{equation}
	for some $x_\eta >0$. 
	
	First we pick $\alpha_1$ so that
		$$\alpha_1 (x) \geq 1  -\frac{\eta}{3},$$
	for all $x \geq 0$, and in particular the first wanted inequality is already satisfied. Then, due to $s < s_1^*$, we may find a subsolution $\underline{u}$ of the scalar homogeneous equation
	   $$u_t = d_1 u_{xx} + r_1 u \left[ 1 -  \frac{2 \eta}{3} - u \right],$$
	whose compact support moves with speed $s$, and such that $\underline{u} (st, t) = 1 - \eta$ for all $t \geq 0$. We omit the details but this can be done by a careful phase plane analysis as in~\cite{aw75}. Up to some shift $x_\eta$, we may assume that its support is included in $\{ x \geq st \}$, and then 
	$$\underline{u} (st + x_\eta, t) \geq 1 - \eta.$$

	Next, for a given $\alpha_2$ and compactly supported $v_0$, one may check that the solution of 
		$$v_t  = d_2  v_{xx} +  r_2  v [\alpha_2 (x-st) - v],$$
	satisfies that, for some $M>0$ large enough,
	$$\sup_{ t \geq 0 , x \leq st -M } v (x,t) \leq \frac{\eta}{3a_2}.$$
	This can be done by means of an exponential supersolution on a left-half line where $\alpha_2$ is negative; we omit the details and refer to the beginning of our proof of Theorem~\ref{thm:lbsf} for a similar argument. Up to shifting $\alpha_2$ and $v_0$, by comparison we find that the solution of the original competition system with climate change satisfies
	$$ v (x,t) \leq \frac{\eta}{3 a_2},$$
	for all $t \geq 0$ and $x$ in the (moving compact) support of $\underline{u}$. Then
	\begin{eqnarray*}
		\underline{u}_t & \leq & d_1 \underline{u}_{xx} + r_1 \underline{u} \left[ 1 - \frac{2 \eta}{3} - \underline{u} \right] 	\\
		& \leq & d_1 \underline{u}_{xx} + r_1 \underline{u} \left[ \alpha_1 (x-st) - a_2 v - \underline{u} \right].
	\end{eqnarray*}
	Up to taking $u_0$ above $\underline{u}$, we conclude by comparison that \eqref{eqn:x_delta} indeed holds true.	
	
	With~\eqref{eqn:x_delta} in hand, letting $\eta$ small enough, and up to increasing (resp. decreasing) $u_0$ (resp. $v_0$), which clearly will preserve \eqref{eqn:x_delta}, we may assume that 
		$$u \geq \underline{U}_\delta, \quad v \leq \overline{V}_\delta,$$
	on the parabolic boundary of the moving set $(st + x_\eta , +\infty)$. As in the end of Section~\ref{sub:subcase1}, a use of the Liouville type Proposition~\ref{prop:entire} for the homogeneous competition system eventually yields that
	$$\lim_{t \to \infty}\left\{\sup_{(s+\varepsilon) t \leq x \leq (s_{1,mtf,\delta}^* - \ep ) t} \Big[ | u(x,t) - 1| + |v(x,t)| \Big] \right\} = 0,$$
	for any small $\ep >0$. Again, since we picked $\delta$ before the initial data, we cannot immediately pass to the limit as $\delta \to 0$. Still, with this in hand, we can now repeat the same argument, proceeding exactly as in Section~\ref{sub:subcase1}, and compare $(u,v)$ with lower-upper solutions $(\underline{U}_\delta, \overline{V}_\delta)$, but now with arbitrarily small $\delta >0$. Recalling that $s_{1,mtf,\delta}^* \to s_{1,GL}^*$ as $\delta \to 0$, this finally ends this proof.


\bigskip

\end{document}